\documentclass{amsart}
\usepackage[utf8]{inputenc}

\usepackage[foot]{amsaddr}

\usepackage{biblatex}
\usepackage{mathtools}
\usepackage{amsmath}
\usepackage{amssymb}
\usepackage{amsthm}
\usepackage{stmaryrd}
\usepackage{hyperref}
\usepackage{cleveref}
\usepackage{listings}
\usepackage{enumitem}

\usepackage[hyphenbreaks]{breakurl}

\usepackage{xcolor}
\usepackage{dutchcal}
\usepackage{tikz-cd}
\usepackage{tikz}
\usetikzlibrary{calc,decorations.markings,decorations.pathmorphing}
\usepackage{ifthen}
\usepackage{bm}

\pgfdeclarelayer{background}
\pgfsetlayers{background,main}

\theoremstyle{definition}
\newtheorem{Def}[]{Definition}[section]
\newtheorem{Pro}[Def]{Proposition}
\newtheorem{Rem}[Def]{Remark}
\newtheorem{Lem}[Def]{Lemma}
\newtheorem{Exp}[Def]{Example}
\newtheorem{Theo}[Def]{Theorem}
\newtheorem{Cor}[Def]{Corollary}

\newtheorem{Con}[Def]{Construction}
\newtheorem{Conv}[Def]{Convention}
\newtheorem*{Ques*}{Question}
\newtheorem*{Theo*}{Theorem}

\DeclareMathOperator{\Ext}{Ext}
\DeclareMathOperator{\Hom}{Hom}
\DeclareMathOperator{\RHom}{\mathbb{R}\strut\kern-.2em\operatorname{Hom}}
\DeclareMathOperator{\Aut}{Aut}
\DeclareMathOperator{\rad}{rad}

\DeclareMathOperator{\Ker}{Ker}
\DeclareMathOperator{\End}{End}
\DeclareMathOperator{\SL}{SL}
\DeclareMathOperator{\GL}{GL}

\DeclareMathOperator{\Db}{\mathrm{D}^{\mathrm{b}}}

\DeclareMathOperator{\Irr}{Irr}

\DeclareMathOperator{\Stab}{Stab}

\DeclareMathOperator{\diag}{diag}

\newcommand{\bbone}{\text{\usefont{U}{bbold}{m}{n}1}}
\MakeRobust{\bbone}

\renewcommand{\mod}{\operatorname{mod}}
\renewcommand{\Bbbk}{\mathbb{C}}
\renewcommand{\Im}{\operatorname{Im}}

\title[$2$-RI algebras from non-abelian subgroups of $\SL_3$. Part I]{$2$-representation infinite algebras from non-abelian subgroups of $\SL_3$\\[0.8em]\smaller{}\smaller{} \text{Part I: Extensions of abelian groups}}

\author[Darius Dramburg]{Darius Dramburg}
\address{\emph{Corresponding author:} Darius Dramburg, Department of Mathematics, Uppsala University, Box 480, 75106 Uppsala, Sweden}

\author[Oleksandra Gasanova]{Oleksandra Gasanova}
\address{Oleksandra Gasanova, Faculty of Mathematics, University of Duisburg-Essen, 45127 Essen, Germany}

\email{darius.dramburg@math.uu.se, oleksandra.gasanova@uni-due.de}

\date{\today}

\begin{document}

\begin{abstract}
    Let $G \leq \SL_3(\mathbb{C})$ be a non-trivial finite group, acting on $R = \mathbb{C}[x_1, x_2, x_3]$. The resulting skew-group algebra $R \ast G$ is $3$-Calabi-Yau, and can sometimes be endowed with the structure of a $3$-preprojective algebra. However, not every such $R \ast G$ admits such a structure. The finite subgroups of $\SL_3(\Bbbk)$ are classified into types (A) to (L). We consider the groups $G$ of types (C) and (D) and determine for each such group whether the algebra $R \ast G$ admits a $3$-preprojective cut, that is a $3$-preprojective structure arising from a grading of the McKay quiver of $G$. We show that the algebra $R \ast G$ admits a $3$-preprojective cut if and only if $9 \mid |G|$. Our proof is constructive and yields a description of the involved $2$-representation infinite algebras. This is based on the semi-direct decomposition $G \simeq N \rtimes K$ for an abelian group $N$, and we show that the existence of a $3$-preprojective structure on $R \ast G$ is essentially determined by the existence of one on $R \ast N$. This provides new classes of $2$-representation infinite algebras, and we discuss some $2$-Auslander-Platzeck-Reiten tilts. Along the way, we give a detailed description of the involved groups and their McKay quivers by iteratively applying skew-group constructions.
\end{abstract}

\maketitle

\section{Introduction}
Let $G \leq \SL_n(\mathbb{C})$ be a (non-trivial) finite group. Then $G$ acts naturally on $\mathbb{C}^n$, and hence on the polynomial ring $R = \mathbb{C}[x_1, \ldots, x_n]$. The resulting skew-group algebra, or twisted tensor product, $R \ast G = R \otimes \mathbb{C}G$ is a prototypical example of an $n$-Calabi-Yau algebra. Such algebras appear in noncommutative geometry, cluster theory, mirror symmetry and mathematical physics. Our interest in this setup stems from the role that certain Calabi-Yau algebras play in Iyama's higher dimensional Auslander-Reiten theory. In the case of skew-group algebras $R \ast G$ as above, this is part of a general McKay correspondence. 

In the classical McKay correspondence \cite{MR0604577}, the finite subgroups $G \leq \SL_2(\mathbb{C})$ are studied via their McKay quivers $Q_G$, which turn out to be doubled versions of the extended Dynkin diagrams of type $\Tilde{A}$, $\Tilde{D}$ or $\Tilde{E}$. The skew-group algebra $R \ast G$ is Morita equivalent to a quotient of the path algebra $\mathbb{C}Q_G/I$ by commutativity relations $I$. One recovers the hereditary tame representation infinite algebras by choosing an acyclic orientation of $Q_G$. Moreover, this choice of orientation determines a locally finite grading of Gorenstein parameter $1$ on $\mathbb{C}Q_G/I$, which turns $\mathbb{C}Q_G/I$ into the preprojective algebra of its degree $0$ part. More precisely, by the work of Donovan-Freislich \cite{DonovanFreislich} and simultaneously by Nazarova \cite{nazarova1973representations}, any basic connected hereditary tame representation infinite algebra is isomorphic to a path algebra $\mathbb{C}Q $ of an acyclic quiver $Q$ whose underlying undirected graph is an extended Dynkin diagram of type $\Tilde{A}$, $\Tilde{D}$ or $\Tilde{E}$. Denote by $D(-) = \Hom_\mathbb{C}(-, \mathbb{C})$ the $\mathbb{C}$-duality. The preprojective algebra \[ \Pi_2(\mathbb{C}Q) = \operatorname{T}_{\mathbb{C}Q} \Ext^1_{\mathbb{C}Q}(D(\mathbb{C}Q), \mathbb{C}Q)\] 
of $\mathbb{C}Q$ is graded by tensor degrees, and is isomorphic to $\mathbb{C}Q_G/I$ for some finite subgroup $G \leq \SL_2(\mathbb{C})$, where the tensor grading on $\Pi_2(\mathbb{C}Q)$ corresponds to a choice of orientation of the underlying extended Dynkin diagram of $Q_G$. 

Note that in the classical case above, $R \ast G$ is $2$-Calabi-Yau, yet the original definitions of Calabi-Yau triangulated categories and Calabi-Yau algebras were, due to their role in string theory and mirror symmetry, $3$-Calabi-Yau. Nowadays, the role of the Serre functor has led to the more general concept of $(n+1)$-Calabi-Yau algebras, and the notion of $(n+1)$-preprojective algebras is avalaible as well. Therefore, it is in general interesting to ask which $(n+1)$-Calabi-Yau algebras can be endowed with an $(n+1)$-preprojective structure, as posed in \cite{Thibault}. 

Yet, the case of finite subgroups of $\SL_3(\Bbbk)$ is particularly interesting for two reasons. From the perspective of higher Auslander-Reiten theory, we are constructing $3$-preprojective algebras of $2$-representation infinite algebras, which is the first ``higher'' dimension than the classical case of $1$-representation infinite algebras. Furthermore, already in dimension $3$ there exist groups $G \leq \SL_3(\mathbb{C})$ for which $R \ast G$ does not admit a $3$-preprojective structure \cite{Thibault}, contrary to the classical case of subgroups of $\SL_2(\mathbb{C})$, and understanding this failure is interesting in its own right.

In this article, we consider skew-group algebras $R \ast G$ where $G \leq \SL_3(\Bbbk)$ is a finite group of so-called type (C) or (D). This terminology of types is standard in the classification of finite subgroups of $\SL_3(\Bbbk)$ \cite{YauYu}. The groups of type (C) and (D) are constructed by choosing some (non-trivial) finite abelian group $A \leq \SL_3(\Bbbk)$, consisting of diagonal matrices, and adding the permutation matrix $t$ of a $3$-cycle and a monomial transposition matrix $r$ to generate $G = \langle A, t \rangle$ of type (C) or $G = \langle A, t, r \rangle $ of type (D). 

To put the types (C) and (D) into perspective, we recall that the classification of finite subgroups of $\SL_3(\Bbbk) $ gives rise to types (A) to (L). Type (A) consists of (diagonal) abelian groups. The existence of higher preprojective structures for type (A) was investigated in dimension $3$ in \cite{DramburgGasanova}, and in arbitrary dimension in \cite{DramburgGasanova2}. Types (C) and (D) arise from type (A) by iterated extensions, and we show that they contain an abelian normal subgroup, which largely controls the behavior of these groups. 
Type (B), in contrast, consists of groups arising from the embedding $\GL_2(\Bbbk) \hookrightarrow \SL_3(\Bbbk)$. The remaining types (E) to (L) are finitely many exceptional cases and are amenable to direct computation. The investigation of types (B) and (E) to (L) is done in \cite{Dramburg1}.  

We determine exactly which groups $G$ of type (C) and (D) give rise to a skew-group algebra $R \ast G$ that admits a $3$-preprojective cut. Cuts are those higher preprojective structures which arise from grading the quiver of $R \ast G$ directly, and we believe that up to isomorphism, no other $3$-preprojective structures can exist on $R \ast G$. The main classification result can be phrased as follows. 
\begin{Theo*}[\Cref{classification}]
    Let $G \leq \SL_3(\Bbbk)$ be a finite group of type (C) or (D). Then $G \simeq N \rtimes K$, where $N$ is abelian, and $K \simeq C_3$ in type (C) and $K \simeq S_3$ in type (D). The skew-group algebra $R \ast G$ admits a $3$-preprojective cut if and only if $3 \mid |N|$, which happens if and only if $9 \mid |G|$.  
\end{Theo*}

This theorem hinges on the semi-direct decomposition $G \simeq N \rtimes K$.  The condition $3 \mid |N|$ is enough to ensure that a $K$-invariant $3$-preprojective cut exists on $R \ast N$. This structure can then be extended to $R \ast G$. This extension procedure is based on computing first the McKay quiver for $N$, and then computing the McKay quiver for $G$ by ``skewing the quiver'' as developed in \cite{ReitenRiedtmann, Demonet, LeMeur}. The same methods do not directly apply to type (B), since these groups do not always admit a similar semi-direct decomposition. 

The same distinction whether $|N|$ is divisible by $3$ has been made in \cite{ItoTrihedral} when computing crepant resolutions and the Euler characteristic of the quotient singularity $\mathbb{C}^3/G$. In particular, the construction of a crepant resolution depends on whether $3 \mid |N|$. Furthermore, one can see that $3 \mid |N|$ is equivalent to requiring that $G \supseteq Z(\SL_3(\Bbbk))$. 

We also point out the following parallel to the McKay correspondence, in particular for type (C) as in \cite{ItoTrihedral}. In type (A), it follows from \cite{DramburgGasanova, DramburgGasanova2} that $3$-preprojective cuts, up to some appropriate equivalence, on $R \ast G$ for abelian $G \leq \SL_n(\mathbb{C})$ are in bijection with compact exceptional crepant toric divisors of $\operatorname{Spec}(R^G)$.
In types (C) and (D), we can decompose $G \simeq N \rtimes K$ so that $N$ is of type (A). Since the $3$-preprojective structures on $R \ast N$ correspond to certain divisors in $\operatorname{Spec}(R^N)$, it would be interesting to understand how the $3$-preprojective structures on $R \ast G$ correspond to crepant divisors in $\operatorname{Spec}(R^G)$, and whether the geometric construction in \cite{ItoTrihedral} mirrors the construction of extending an $K$-invariant $3$-preprojective structure on $R \ast N$ to $R \ast G$.

\subsection{Outline}
In \Cref{Sec: Preliminaries}, we collect the necessary background material on quivers, higher Auslander-Reiten theory and skew-group algebras. 

We summarise our results on abelian groups in $\SL_3(\Bbbk)$ from \cite{DramburgGasanova} in \Cref{Sec: type (A)} since they will be needed for the other cases. 

In \Cref{Sec: (C) and (D)}, we cover the groups of types (C) and (D) together because many of the results can be developed in parallel. We describe the groups $G$, as well as their semi-direct decomposition $G \simeq N \rtimes K$ in \Cref{SSec: Group structure}. This allows us to compute their McKay quivers and cuts in \Cref{SSec: McKay and cuts}. Our main result is \Cref{classification}, showing that $3$-preprojective cuts on $R \ast G$ exist if and only if $3 \mid |N|$. We conclude by discussing in \Cref{SSec: Skewing and unskewing} how to transfer certain invariant $3$-preprojective gradings from $R \ast N$ to $R \ast G$ and back.  

\section{Preliminaries}\label{Sec: Preliminaries}
\subsection{Conventions and Notation}
We work over the field $\mathbb{C}$ throughout, but mention that the same results hold over any algebraically closed field of characteristic $0$. In this article, all undecorated tensors $\otimes$ are taken over $\Bbbk$. We denote isomorphism by $\simeq$ and Morita equivalence by $\simeq_M$. For $a,b \in \mathbb{Z}$ and $n \in \mathbb{N}$ we write $a \equiv_n b$ if the residue classes $ a + n\mathbb{Z} = b + n \mathbb{Z}$ agree. We write $a \bmod n$ for the smallest nonnegative representative of the class $a + n \mathbb{Z}$. 

Let $\epsilon$ be a primitive $n$-th root of unity, then we denote by $\frac{1}{n}(a_1, \ldots, a_m)$ the diagonal matrix $\operatorname{diag}(\epsilon^{a_1} , \ldots, \epsilon^{a_m})$. While the notation depends on which primitive $n$-th root is chosen, we often simply write $\frac{1}{n}(a_1, \ldots, a_m)$ when the choice does not matter. 
 
\subsection{Quivers for nonnegatively graded algebras}
Let $\Gamma = \bigoplus_{i \geq 0} \Gamma_i$ be a nonnegatively graded algebra. We are mainly interested in \emph{locally finite} graded algebras, that means graded algebras for which $\dim(\Gamma_i)$ is finite for each $i$. We will often want to speak of \emph{the quiver} of $\Gamma$, but since $\Gamma$ can be infinite dimensional we need to take some care. For a natural number $n$, we denote by $\Gamma_{\geq n} = \bigoplus_{i \geq n} \Gamma_i$ the component of $\Gamma$ of degree at least $n$. The easiest definition of the quiver we want to talk about is given by simply truncating $\Gamma$.  

\begin{Def}
    Let $\Gamma = \bigoplus_{i \geq 0} \Gamma_i$ be a nonnegatively graded algebra, generated in degrees $0$ and $1$, and such that $\Gamma_0$ and $\Gamma_1$ are finite dimensional. Then the quiver of $\Gamma$ is defined to be the quiver of the finite dimensional algebra $\Gamma/\Gamma_{\geq 2}$.  
\end{Def}

Let us now justify this definition. To start, note that $\Gamma_1$ is a $\Gamma_0$-bimodule, so we can form the tensor algebra 
\[ \operatorname{T}_{\Gamma_0}( \Gamma_1). \]
This algebra is graded by tensor degree, and we obtain an obvious surjection 
\[ \operatorname{T}_{\Gamma_0} (\Gamma_1) \to \Gamma. \]
With the correct basicness assumptions, this setup produces the quiver and the ideal we expect. 

\begin{Con}\label{Con: Quiver for graded alg}
    Suppose that $\Gamma_0$ is basic. Since it is finite dimensional, we have  $\Gamma_0 \simeq \Bbbk(Q')/I'$ for a unique quiver $Q'$ and admissible ideal $I'$. We choose a minimal set $S$ of generators of $\Gamma_1/\rad_{\Gamma_0^e}(\Gamma_1)$ as a $\Gamma_0/\rad_{\Gamma_0^e}(\Gamma_0)$-bimodule. Furthermore, we choose this set so that it is compatible with the action of $\Bbbk Q_0$. More precisely, we denote by $e_i$ the idempotent corresponding to the vertex $i \in Q_0$, and choose $S$ so that it can be partitioned into $S = \bigcup_{i,j} S_{i,j}$ such that $e_i S_{i',j'} e_j = S_{i,j}$ if and only if $i= i'$ and $j=j'$, and $e_i S_{i',j'} e_j = \{ 0 \}$ otherwise. Since the enveloping algebra $\Gamma_0^e$ is finite dimensional, we can apply Nakayama's lemma to the finite dimensional $\Gamma_0$-bimodule $\Gamma_1$ to lift $S$ modulo the radical to a generating set $S'$ of $\Gamma_1$ as a $\Gamma_0$-module. We add these generators $S'$ as new arrows to $Q'$ to obtain a quiver $Q \supseteq Q'$. Hence we obtain a surjection $\Bbbk Q \to \Gamma$, taking $\Bbbk Q'$ to $\Gamma_0$, and the new arrows to the corresponding generators in $\Gamma_1$. We grade $\Bbbk Q$ by placing $Q'$ in degree $0$ and the new arrows in degree $1$. This way, the kernel $I = \Ker(\Bbbk Q \twoheadrightarrow \Gamma)$ becomes a homogeneous ideal in $\Bbbk Q$, and we have $\Bbbk Q/I \simeq \Gamma$ as graded algebras. Furthermore, $I$ is ``admissible'' in the sense that $I_0 = I'$, and that $I \subseteq \langle Q_1 \rangle^2 $, since $S$ was a minimal generating set. The construction only depends on $\Gamma_0$ and $\Gamma_1$, hence it is easy to see that $Q$ is the same quiver as the one of $\Gamma/\Gamma_{\geq 2}$. 
\end{Con}

\begin{Pro}\label{Pro: Quivers for graded algebras}
    Let $\Gamma_0$ be basic and $Q \supseteq Q'$ be as in \Cref{Con: Quiver for graded alg}. Then $Q$ is the same quiver as that of $\Gamma/\Gamma_{\geq 2}$. 
\end{Pro}

We note that this notion is also compatible with graded Morita equivalence. 

\begin{Rem}
    Let $e \in \Gamma_0$ be a Morita idempotent. Then $e\Gamma e = \bigoplus_{i \geq 0} e \Gamma_i e$ is again a graded algebra, and $\Gamma$ and $e \Gamma e$ are graded Morita equivalent in the sense of \cite{GradedMorita}. In particular, we can choose $e \in \Gamma_0$ so that $e \Gamma_0 e$ is basic, and hence $\Gamma$ is graded Morita equivalent to a quotient of a quiver algebra as in \Cref{Pro: Quivers for graded algebras}.    
\end{Rem}

Let us conclude by pointing out that if $\Gamma$ is finite dimensional, then the discussion reproduces the known finite dimensional situation, since then the Jacobson radical is $\rad(\Gamma) = \rad(\Gamma_0) \oplus \Gamma_{\geq 1}$. 

\subsection{Higher representation infinite algebras}
Let $\Lambda$ be a finite dimensional $\Bbbk$-algebra of finite global dimension $d \geq 1$. We remind the reader of the following functors:
\[ \nu = D\RHom_\Lambda(\--,\Lambda) \colon \Db( \mod \Lambda) \to \Db( \mod \Lambda)  \] 
is the \emph{derived Nakayama functor} on the bounded derived category of finitely generated $\Lambda$-modules, and its quasi-inverse is
\[\nu^{-1} = \RHom_{\Lambda^{\mathrm{op}}}(D(\--), \Lambda) \colon \Db( \mod \Lambda) \to \Db( \mod \Lambda).\]
The \emph{derived higher Auslander-Reiten translation} is the autoequivalence 
\[\nu_d := \nu \circ [-d] \colon \Db(\mod \Lambda) \to \Db( \mod \Lambda).\]

\begin{Def}\cite[Definition 2.7]{HIO} 
The algebra $\Lambda$ is called \emph{$d$-representation infinite} if for any projective module $P$ in $ \mod \Lambda$ and integer $i \geq 0$ we have 
\[ \nu_d^{-i} P \in   \mod \Lambda. \]
\end{Def}

While these algebras are important in Iyama's higher Auslander-Reiten theory, we are interested in them due to their connection to $(d+1)$-Calabi-Yau algebras via the \emph{higher preprojective algebra}.

\begin{Def}\cite[Definition 2.11]{IyamaOppermannStable}
Let $\Lambda$ be of global dimension at most $d$. The $(d+1)$-preprojective algebra of $\Lambda$ is 
\[ \Pi_{d+1}(\Lambda) = \operatorname{T}_\Lambda \Ext^{d}_\Lambda(D(\Lambda), \Lambda). \]
\end{Def}

As we will see soon, if $\Lambda$ is $d$-representation infinite, its $(d+1)$-preprojective algebra is $(d+1)$-Calabi-Yau.
Furthermore, note that $\Pi_{d+1}(\Lambda)$ naturally comes with a grading induced from tensor degrees, and that we recover $\Lambda$ as the degree $0$ part of this grading. The properties of this grading, together with being Calabi-Yau, essentially determine the preprojective algebras of higher representation infinite algebras. 

\begin{Def}\cite[Definition 3.1]{AIR} \label{Def: n-CY GP a}
Let $\Gamma=\bigoplus_{i \geq 0}\Gamma_i$ be a positively graded $\Bbbk$-algebra. We call $\Gamma$ a \emph{bimodule $(d+1)$-Calabi-Yau algebra of Gorenstein parameter $a$} if there exists a bounded graded projective $\Gamma$-bimodule resolution
$P_\bullet$ of $\Gamma$ and an isomorphism of complexes of graded $\Gamma$-bimodules
\[ P_\bullet \simeq \Hom_{\Gamma^{\mathrm{e}}}(P_\bullet, \Gamma^{\mathrm{e}})[d+1](-a).  \]
\end{Def}

Dropping the grading from this definition, we recover the definition of a bimodule $(d+1)$-Calabi-Yau algebra. When introducing different gradings on $\Gamma$ so that the projective bimodule resolution becomes graded, the resulting Gorenstein parameter varies. We are interested precisely in the case $a=1$.

\begin{Theo}\cite[Theorem 4.35]{HIO}\label{Theo: HPG is f.d. GP1}
There is a bijection between $d$-representation infinite algebras $\Lambda$ and graded bimodule $(d+1)$-Calabi-Yau algebras $\Gamma$ of Gorenstein parameter 1 with $\dim_\Bbbk (\Gamma_i) < \infty $ for all $i \in \mathbb{N}$, both sides taken up to isomorphism. The bijection is given by 
\[ \Lambda \mapsto \Pi_{d+1}(A) \quad \mathrm{ and } \quad \Gamma \mapsto \Gamma_0.  \]
\end{Theo}

Hence we take the perspective that admitting a higher preprojective grading is a property of a Calabi-Yau algebra. The question whether a given Calabi-Yau algebra has this property was raised in \cite{Thibault}. We therefore need a source of Calabi-Yau algebras, which leads to our next subsection.

\subsection{Skew-group algebras}
Let $G$ be a finite group acting from the left on some $\Bbbk$-algebra $R$ via automorphisms. Then the \emph{skew-group algebra} of $R$ by $G$ is the vector space 
\[ R \ast G = R \otimes \Bbbk G, \]
with multiplication induced from 
\[ (r \otimes g) (s \otimes h) = r g(s) \otimes gh. \]

In this article, we consider the case where $G \leq \SL_{d+1}(\Bbbk)$ is finite. Then $G$ acts naturally on $\Bbbk^{d+1}$, and hence on the polynomial ring $R = \Bbbk [x_1, \ldots, x_{d+1}]$. We fix this notation throughout. 

\begin{Pro}\cite[Theorem 3.2]{BSW}
    Let $G \leq \SL_{d+1}(\Bbbk)$ be finite, acting on the polynomial ring $R = \Bbbk [x_1, \ldots, x_{d+1}]$. Then $R \ast G$ is $(d+1)$-Calabi-Yau.  
\end{Pro}

The following description of the quiver and superpotential for $R \ast G$ is based on the fact that $R$, and hence $R \ast G$, is Koszul with respect to the standard polynomial grading. We also need the notion of a McKay quiver. 

\begin{Def}
    Let $G \leq \SL_{d+1}(\Bbbk)$ be finite. Denote the natural $(d+1)$-dimensional representation of $G$ by $\rho$. The McKay quiver $Q = Q_G$ has as vertices the irreducible representations of $G$, i.e.\ $Q_0 = \Irr(G)$, and for two irreducible representations $\chi_i, \chi_j \in Q_0$, the arrows from $\chi_i$ to $\chi_j$ are a basis of $\Hom_{\Bbbk G}(\chi_i , \chi_j \otimes \rho)$. 
\end{Def}

For the details of the description of the superpotential $\omega$ in the following theorem, we refer the reader to \cite{BSW}. 

\begin{Theo}\cite[Lemma 3.1, Theorem 3.2]{BSW}
    Let $G \leq \SL_{d+1}(\Bbbk)$ be finite, acting on the polynomial ring $R = \Bbbk [x_1, \ldots, x_{d+1}]$. Then $R \ast G$ is Morita equivalent to a basic algebra $\Bbbk Q_G/I$. The quiver $Q_G$ is the McKay quiver of $G$, and the ideal of relations $I \subseteq (\Bbbk(Q_G)_1)^2$ is induced from the commutativity relations in $R$. More precisely, $I = \langle \partial_{d-1} \omega \rangle$ is generated by the $(d-1)$-th derivatives of a superpotential $\omega$ on $Q_G$, and $\omega$ is a linear combination of cycles of length $d+1$ in $Q_G.$
\end{Theo}

Thus, we have a large class of Calabi-Yau algebras available, and we ask when such a Calabi-Yau algebra $R \ast G$ arises as a higher preprojective algebra. More precisely, we are interested in constructing a locally finite grading of Gorenstein parameter $1$ on $R \ast G$. There is to date no general characterization when $R \ast G$ admits such a grading. However, there are large classes of examples known, such as those coming from certain cyclic or metacyclic groups \cite{AIR, Giovannini}, and classes of counterexamples coming from direct product decompositions \cite{Thibault}, as well as a classification of those abelian groups $G \leq \SL_{d+1}(\Bbbk)$ for which $R \ast G$ admits such a grading \cite{DramburgGasanova2}.   

\subsection{Higher preprojective cuts}
We are interested in the existence of higher preprojective gradings on skew-group algebras $R \ast G$ for a finite group $G \leq \SL_{d+1}(\Bbbk)$ acting on the polynomial ring $R$ in $d+1$ variables. The precise relationship between the quiver of $R \ast G$ and such gradings is subtle, so we summarise the relevant results here. We begin by noting how to transfer higher preprojective gradings along Morita equivalences. 

\begin{Rem}
    Let $e \in \Gamma$ be a Morita idempotent. If $\Gamma$ is equipped with a higher preprojective grading so that $e$ is homogeneous of degree $0$, then the induced grading on $ e \Gamma e$ is also higher preprojective. Conversely, if $e \Gamma e$ is equipped with a higher preprojective grading, then $e \Gamma$ is a gradeable $e \Gamma e$-module and $\Gamma \simeq \End_{e \Gamma e}(e \Gamma)$ inherits a higher preprojective grading. In particular, this means that an algebra $\Gamma$ admits a higher preprojective grading if and only if the Morita equivalent quotient of a path algebra $\Bbbk Q/I$ admits a higher preprojective grading. 
\end{Rem}

Hence, in order to decide whether $R \ast G$ admits a higher preprojective grading, it suffices to check this for the Morita equivalent algebra $\Bbbk Q_G/I$, where $Q_G$ is the McKay quiver of $G$. When constructing such a grading on $\Bbbk Q_G/I$, it is tempting to grade $Q_G$ directly. Such gradings on $Q_G$ are called \emph{cuts}.

\begin{Def}
    A higher preprojective grading on a basic Calabi-Yau algebra $\Bbbk Q/I$ is called a \emph{cut} if all arrows of $Q$ are homogeneous of degree $0$ or $1$, and all vertices in $Q$ are homogeneous of degree $0$. We refer to the induced grading of $Q$ as the cut on $Q$.
\end{Def}

Let us now discuss the relevance of cuts. 

\begin{Rem}
    Suppose that the Calabi-Yau algebra $\Bbbk Q/I$ is graded with respect to path-length, and Koszul with respect to this grading. Further, suppose there exists a higher preprojective grading on $\Bbbk Q/I$ for which the resulting higher representation infinite algebra $\Lambda$ has acyclic Gabriel quiver. It was proven in \cite{DramburgSandoy} that there exists an automorphism of $\Bbbk Q/I$ mapping the higher preprojective grading to a higher preprojective cut. 
    In light of the recent discovery by Tomonaga \cite[Example 6.5]{TomonagaSiltingMutations} of a $2$-representation infinite algebra which has a cycle in its Gabriel quiver, we note that it is therefore not clear whether it suffices to consider cuts to exhaust all possible higher preprojective gradings. Note that the example in \cite{TomonagaSiltingMutations} is not Koszul, and we believe that the Koszul assumption on $\Bbbk Q/I$ forces any possible $A$ to have acyclic Gabriel quiver. 
\end{Rem}

From now on, we only consider higher preprojective cuts. It follows from our results that all possible cuts on the algebras we consider indeed have acyclic degree $0$, so a higher preprojective grading not equivalent to a cut would necessarily have to be constructed in a different way. 

We will construct cuts by iterated skewing, but we also need a criterion to exclude those algebras which do not admit a cut. We will use the following easy observation to conclude that no higher preprojective cut exists for the skew-group algebras we consider in this article. 

\begin{Pro}\label{Pro: loops in support}
    Let $\Bbbk Q/I$ be a $(d+1)$-Calabi-Yau algebra for which $I = \langle \partial_{d-1} \omega \rangle $ is generated by the $(d-1)$-th derivatives of a superpotential $\omega$. If $Q$ contains a loop $c$ such that some power $c^i$ for $i \geq 2$ is a summand in $\omega$, then $\Bbbk Q/I$ does not admit a higher preprojective cut.
\end{Pro}

\begin{proof}
    By \cite[Lemma 4.9]{Thibault}, a higher preprojective cut on $\Bbbk Q/I$ extends to a grading on $\Bbbk Q$, and by \cite[Proposition 3.14]{DramburgGasanova} the (super)potential $\omega$ is homogeneous of degree $1$. Let $\Lambda$ be the $n$-representation infinite algebra in degree $0$ defined by the cut. Then the Gabriel quiver of $\Lambda$ is a subquiver of $Q$. By the No-Loop theorem \cite{NoLoops}, the quiver of $\Lambda$ can not contain $c$. Thus, $c$ has to be in degree at least $1$. But then the summand $c^i$ of $\omega$ has degree at least $2$, which contradicts the fact that $\omega$ has degree $1$. 
\end{proof}

\subsection{Skewed quivers and skewed cuts}
In this section, we summarise how to compute the quiver $Q \ast G$ of a skew-group algebra $(\Bbbk Q/I) \ast G$ from the quiver $Q$. In \cite{ReitenRiedtmann}, Reiten and Riedtmann described the quiver $Q \ast G$ in case $G$ is cyclic. For our applications this case suffices, but we give a more general summary for an arbitrary finite group $G$. This is a result of Demonet \cite{Demonet}. 
We begin with group actions on path algebras. 

\begin{Conv}\label{Conv: Action on quiver}
    Let $G$ act on $\Bbbk Q$ such that $G$ permutes the primitive idempotents $\{ e_i \mid i  \in Q_0 \}$ and such that $G$ fixes the space of arrows $\Bbbk(Q_1)$. For each idempotent, we denote its stabiliser by $G_i = \operatorname{Stab}_G(e_i)$, and the intersection of two stabilisers as $G_{i,j} = \Stab_G(e_i) \cap \Stab_G(e_j)$. Since $G$ acts by algebra automorphisms, it follows that each $e_i (\Bbbk Q_1) e_j $ can be seen as a $\Bbbk G_{i,j}$-module, which we denote by $\rho_{i,j}$. 
    Denote by $S$ a transversal for the action of $G$ on $\{e_i \mid i \in Q_0\} $, and denote by $s_i \in S$ the representative of the $G$-orbit of $e_i$. Next, for each $(s_i, s_j) \in S \times S$, consider the action of $G$ on the product of orbits $(G \cdot e_i) \times (G \cdot e_j)$, and fix a transversal $S_{i,j}$ for this action.

    Lastly, for each $e_j \in G \cdot s_i$, fix an element $g_{s_i \to e_j}$ such that $g_{s_i \to e_j}(s_i) = e_j$. In the following, the notation $(-)^g$ denotes conjugation by the element $g \in G$. 
\end{Conv}

\begin{Def}
\label{Def: skewed}
    Let $G$ act on $\Bbbk Q$ as in \Cref{Conv: Action on quiver}, and define the quiver $Q \ast G$ as follows. The vertices are 
    \[ (Q \ast G)_0 = \bigcup_{e_i \in S} \{e_i \} \times \Irr(\Stab_G(e_i)).  \]
    For vertices $(e_i, \varphi), (e_j, \psi) \in (Q \ast G)_0$, the arrows from $(e_i, \varphi)$ to $(e_j, \psi)$ are a basis of 
    \[ \bigoplus_{(u_1, u_2) \in S_{i,j}} \Hom_{\Bbbk G_{u_1,u_2}} \left( \varphi^{g_{s_i \to u_1}}_{|G_{u_1,u_2}}, \psi^{g_{s_j \to u_2}}_{|G_{u_1,u_2}} \otimes_k \rho_{u_1, u_2} \right) .  \]
\end{Def}

\begin{Theo}\cite[Theorem 1]{Demonet} \label{Theo: Demonet}
     With the setup from \Cref{Conv: Action on quiver}, the algebras $(\Bbbk Q) \ast G$ and $\Bbbk(Q \ast G)$ are Morita equivalent. 
\end{Theo}

We note one immediate consequence, which will be useful when combined with \Cref{Pro: loops in support}. 

\begin{Cor}\label{Cor: Full connected orbit implies loop}
    Let $G$ act on $\Bbbk Q$ as above. If there exists a vertex $e \in Q_0$ such that the orbit $G \cdot e$ has size $|G|$ and such that there exists an arrow $e \to g(e)$ for some $g \in G$, then the quiver $Q \ast G$ has a loop. 
\end{Cor}

\begin{proof}
    Without loss of generality, write $e_i = e$. Note that in this situation, the orbit $G \cdot e$ gives rise to a single vertex $(e_i, \mathbf{1})$ since the stabilizer $\operatorname{Stab}_G(e_i)$ is trivial. Furthermore, the arrow $e \to g(e)$ gives rise to some non-zero $\rho_{u_1, u_2}$ for some $(u_1, u_2) \in S_{i,i}$, and hence the space defining the loops at the vertex $(e_i, \mathbf{1})$ is non-zero. 
\end{proof}

Next, we note that the above discussion can be extended to quotients of path algebras by $G$-stable ideals. 

\begin{Rem}
    Let the finite group $G$ act on the algebra $\Bbbk Q$ as in \Cref{Conv: Action on quiver}. Let $I \subseteq (\Bbbk Q_1)^2$ be an ideal such that $G \cdot I = I$. Then $G$ also acts on $\Bbbk Q/I$, and the algebra $(\Bbbk Q/I) \ast G$ is Morita equivalent to a quotient $\Bbbk( Q \ast G)/I'$ for an ideal $I' \subseteq (\Bbbk( Q \ast G)_1)^2$. 
\end{Rem}

Since we want to transport higher preprojective cuts along skew-group constructions, we need the following theorem of Le Meur.

\begin{Theo}\cite[Proposition 6.2.1]{LeMeur}
    Let $\Lambda$ be $d$-representation infinite, and let the finite group $G \leq \Aut(\Lambda)$ act on $\Lambda$. This induces an action of $G$ on $\Pi_{d+1}(\Lambda)$, and $\Lambda \ast G$ is $d$-representation infinite with 
    \[ \Pi_{d+1}(\Lambda \ast G ) \simeq \Pi_{d+1}(\Lambda) \ast G.  \]
\end{Theo}

We will often use the following corollary. 

\begin{Cor}\label{Cor: Skewing by graded auts}
    Let $G \leq \operatorname{GrAut(\Pi_{d+1}(\Lambda))}$ be a finite subgroup of graded automorphisms. Then $\Pi_{d+1}(\Lambda) \ast G$ is a graded algebra, and the grading is again higher preprojective.  
\end{Cor}

This leads us to the following definition. 

\begin{Def}
    Let $C \subseteq Q_1$ be a higher preprojective cut on $\Bbbk Q/I$, and $G \leq \operatorname{GrAut}(\Bbbk Q/I)$ a finite subgroup acting on $\Bbbk Q$ as in \Cref{Conv: Action on quiver}. Consider the algebra $ \Bbbk(Q \ast G)/I'$ which is Morita equivalent to $(\Bbbk Q/I) \ast G $. Then $C$ induces a higher preprojective cut on $\Bbbk(Q \ast G)/I'$ which is denoted by $C \ast G$.        
\end{Def}

\section{Type (A)}\label{Sec: type (A)}
This type consists of the abelian subgroups of $\SL_3(\Bbbk)$. A construction and classification of cuts for this case has been performed in \cite{DramburgGasanova}, and generalized to the abelian subgroups of $\SL_d(\Bbbk)$ in \cite{DramburgGasanova2}. We summarize the necessary parts of these articles here. Let $G \leq \SL_3(\Bbbk)$ be finite abelian, denote its order by $n = |G|$, and denote the defining representation by $\rho \colon G \hookrightarrow \SL_3(\Bbbk)$. The fact that $\rho$ is faithful is equivalent to the fact that its irreducible summands $\rho = \rho_1 \oplus \rho_2 \oplus \rho_3$ generate the dual group $\hat{G} = \Hom( G, \mathbb{C}^\ast)$, and we can view $Q_G$ as the Cayley graph of $\hat{G}$ with respect to this generating set. 

To give a more concise description of the McKay quiver $Q = Q_G$, we first define a universal cover. Fix the lattice $\mathbb{Z}^2$ with the basis $e_1, e_2$, and fix a third vector $e_3 = -(e_1 + e_2)$. The infinite quiver $\hat{Q}$ is defined via 
\begin{align*}
    \hat{Q}_0 = \mathbb{Z}^2, \quad \hat{Q}_1 = \{ (x \to x + e_i) \mid 1 \leq i \leq 3 \}.
\end{align*}

Next, consider the homomorphism
\[ q \colon \mathbb{Z}^2 \to \hat{G}, e_i \mapsto \rho_i,  \]
and denote its kernel by $\Ker(q) = L$. Both $\mathbb{Z}^2$ and $L$ act on $\hat{Q}$ in an obvious way, it is easy to see that we can identify
\[ Q = Q_G = \hat{Q}/L.  \]

Furthermore, in this way, each arrow $a$ in $\hat{Q}$ and hence in $Q$ can be given a unique \emph{type} $\theta(a) = i$, which is the number $i \in \{ 1, 2,3\}$ such that $a = (x \to x + e_i)$. A cut $C \subseteq Q_1$ then has the \emph{type}
\[ \theta(C) = (\# \{ a \in C \mid  \theta(a) = i \} )_{ 1 \leq i \leq 3}. \]

Using the language of types, we can identify the \emph{elementary cycles}. Those are the $3$-cycles in $Q$ which consist of arrows of $3$ distinct types. 

The last piece of notation we need is the specific embedding $ L  \xhookrightarrow{B} \mathbb{Z}^2$, where we view $B \in \mathbb{Z}^{2 \times 2}$ as a matrix with respect to the basis $e_1, e_2$ of $\mathbb{Z}^2$ and an arbitrary basis of $L$. 

\begin{Rem} \label{Rem: Shape of B matrix}
    The matrix $B$ is only unique up to right $\operatorname{GL}_2(\mathbb{Z})$-multiplication, corresponding to a change of basis for $L$. It will be convenient to choose certain nice representatives, so we note that we choose $B$ without loss of generality so that $\det(B) = |G|$. Furthermore, we can choose 
    \[ B = \left( \begin{matrix}
    a & b \\ 0 & c
    \end{matrix} \right), \] 
    to be upper triangular with only nonnegative entries and such that $b < a$, by considering the Hermite normal form. 
\end{Rem}

\begin{Theo}\cite[Theorem 7.7, Theorem 7.12]{DramburgGasanova} \label{Theo: SL3 type (A) classification}
    Let $\gamma \in \mathbb{Z}^{1 \times 3}$. Then there exists a higher preprojective cut $C$ on $Q = Q_G$ such that $\theta(C) = \gamma$ if and only if $\gamma_1 + \gamma_2 + \gamma_3 = n = |G|$ and $\gamma_i > 0$ for all $i$ and 
    \[ (\begin{matrix}
    \gamma_1 & \gamma_2 
    \end{matrix}) \cdot  B  \in n \mathbb{Z}^{1 \times 2}. \]
    Furthermore, if $C$ is a higher preprojective cut, then $C$ is mutation equivalent to all cuts $C'$ with $\theta(C) = \theta(C')$. 
\end{Theo}

The constructive part will be useful when constructing cuts for groups of type (C) and (D), so we summarize it here. Given a vector $\gamma \in \mathbb{Z}^{1 \times 3}$ such that $\gamma_1 + \gamma_2 + \gamma_3 = n$ and $\gamma_i > 0$ for all $i$, and 
\[ (\begin{matrix}
    \gamma_1 & \gamma_2 
    \end{matrix}) \cdot  B  \in n \mathbb{Z}^{1 \times 2}. \]
consider the map  
\[ \xi_\gamma \colon \mathbb{Z}^2 \to \mathbb{Z}/n\mathbb{Z}, (x_1e_1 + x_2 e_2) \mapsto (\gamma_1 x_1 + \gamma_2 x_2) + n\mathbb{Z}.   \]
Note that the kernel $L' = \Ker(\xi_\gamma)$ contains $L$. The image of this map is a cyclic group of order $n' = \frac{n}{\gcd(\gamma_1, \gamma_2, \gamma_3)}$, so we obtain an induced isomorphism 
\[ \overline{\xi_\gamma} \colon \mathbb{Z}^2/L' \to \mathbb{Z}/n'\mathbb{Z}.  \]
We then construct a cut $C_\gamma$ as follows. An arrow $a \colon (x_1e_1 + x_2e_2) + L \to (x_1e_1 + x_2 e_2 + e_i) + L$ is in $C_\gamma$ if and only if 
\[ \overline{\xi_\gamma}((x_1e_1 + x_2e_2) + L') > \overline{\xi_\gamma}((x_1e_1 + x_2e_2 + e_i) + L'), \]
where the inequality is taken for the smallest nonnegative representatives of the elements in $\mathbb{Z}/n'\mathbb{Z}$. 

\begin{Theo}\cite[Proposition 7.30]{DramburgGasanova} \label{Theo: Formula for cut}
    With the setup above, the set $C_\gamma$ is a higher preprojective cut of type $\gamma$ on $Q$. 
\end{Theo}

\section{Types (C) and (D)}\label{Sec: (C) and (D)}
In this section, we investigate the subgroups of $\SL_3(\Bbbk)$ of type (C) and (D). Type (C) is sometimes called \emph{trihedral} \cite{ItoTrihedral}. Following the presentation of Yau and Yu \cite{YauYu}, the groups of type (C) and (D) are generated by a \emph{non-trivial} diagonal group $A$ of type (A), and one (respectively, two) more matrices. We change notation from Yau and Yu slightly, and define 
\begin{align*}
t = \left( \begin{smallmatrix} 0 & 0 & 1 \\ 1 & 0 & 0 \\ 0 & 1 & 0 \end{smallmatrix} \right)  \quad   r = \left( \begin{smallmatrix}
    0 & \alpha & 0 \\  \beta & 0 & 0 \\ 0 & 0 & \gamma
\end{smallmatrix} \right), 
\end{align*}
where $\alpha\beta\gamma = -1$. 
Then a finite subgroup $G \leq \SL_3(\Bbbk)$ for which $V = \Bbbk^3$ is an irreducible representation is said to be of type (C) if 
\[ G = \langle A, t \rangle \]
and of type (D) if 
\[ G = \langle A, t, r \rangle. \]
These groups can be found in $\SL_3(\Bbbk)$ by considering the subgroups $G \leq \SL_3(\Bbbk)$ for which $V = \Bbbk^3$ is an irreducible, but imprimitive representation. Being imprimitive means that $V = V_1 \oplus V_2 \oplus V_3$ decomposes into subspaces such that for every $g \in G$ there exists a permutation $\sigma \in S_3$ such that $g(V_i) = V_{\sigma(i)}$. This gives rise to a homomorphism $G \to S_3$, and since $V$ is irreducible the image of this homomorphism is a transitive permutation group. In type (C), the stabilizer in $S_3$ of each space $V_i$ is trivial, while in type (D) it is not. Clearly, every group of type (D) contains a group of type (C). In the following, we fix $T = \langle t \rangle $ and $H = \langle t,r \rangle$. 

It is important to point out that we always assume that $V = \Bbbk^3$ is an irreducible representation of $G$, which puts some mild restrictions on which $A$ we can consider. The following is essentially the only edge case that can appear. 

\begin{Exp}
    Let $A = \langle \frac{1}{3} (1,1,1) \rangle$ be cyclic of order $3$. Then the group $G = \langle A, t \rangle$ is abelian, since $T$ acts trivially on $A$ by conjugation. Thus, $V = \Bbbk^3$ is not an irreducible representation of $G$, and hence $G$ is not of type (C). 
\end{Exp}

Before we continue, we make some easy observations about the values $\alpha$, $\beta$ and $\gamma$ that can occur in $r$. 

\begin{Rem}
    Consider the group $H = \langle t,r \rangle$. Since $H \leq G$ is finite, all its elements have finite order. In particular, the element $r^2$, which is a diagonal matrix with $\gamma^2$ as one of the entries, has finite order. Hence we conclude that $\gamma$ is a root of unity. Similarly, the element $(tr)^2$ is a diagonal matrix with $\alpha^2$ as one of the entries, proving that $\alpha$ is a root of unity. Since $\alpha \beta \gamma = 1$, the three values $\alpha$, $\beta$ and $\gamma$ are all roots of unity.
\end{Rem}

\subsection{The group structure}\label{SSec: Group structure}
Now we investigate the group structure. As before, we consider a group $G =  \langle A, t \rangle $ of type (C) or $G = \langle A, t, r \rangle$ of type (D). Much of the structure can be described in parallel, so we only specify the type when necessary. 

It is important to point out that while $A \leq G$ is an abelian subgroup, it is in general not normal. Phrased differently, the groups $T, H \leq G$ do not in general act on arbitrary $A$. Instead, $G$ contains a larger abelian normal subgroup $A \leq N $ which we now describe. Its structure is limited by the fact that $T$ and $H$ act on it by conjugation. 

\begin{Lem}
    Let $G$ be of type (C) or (D). Then $N = \{ g \in G \mid g \text{ is diagonal} \}$ is an abelian normal subgroup of $G$.
\end{Lem}

\begin{proof}
    It is easy to see that the subgroup consisting of diagonal matrices is abelian. So see that it is normal, note that conjugation by $t$ or $r$ amounts to permuting diagonal entries.
\end{proof}

Note that in type (C), we have $N = A^T$, and this leads to a semi-direct decomposition. Replicating the same situation for type (D) involves more work since generating $N$ is more involved. Indeed, the matrix $r$ can square to a non-trivial diagonal matrix. Therefore, we need to replace $H$ by a different group, so that we can decompose $G$ as a semi-direct product. 

\begin{Lem}\label{Lem: Semi-direct decomposition}
    Let $G$ be of type (C) or (D). Then $G = N \rtimes K$ is a semi-direct product. In type (C), the complement is $K = T$, and in type (D) the complement is $K \simeq S_3$. 
\end{Lem}

\begin{proof}
    In type (C), the statement is obvious since $T$ contains no diagonal matrices and $G = NT$. 

    In type (D), we find an explicit complement of $N$ in $G$. Let $i_1 = t r^2 t^{-1}r$ and $i_2 = t^2 r^2 t^{-1} r t^{-1}$. One can explicitly compute that 
    \begin{align*}
     i_1= \left( \begin{smallmatrix}
    0 & \alpha\gamma^2 & 0 \\  \alpha\beta^2 & 0 & 0 \\ 0 & 0 & -1
\end{smallmatrix} \right),  \quad   i_2= \left( \begin{smallmatrix}
    -1 & 0 & 0 \\  0 & 0 & \alpha\gamma^2 \\ 0 & \alpha\beta^2 & 0
\end{smallmatrix} \right).
    \end{align*}
    One can also compute that 
    \[ i_1i_2=\left( \begin{smallmatrix}
    0 & 0 & \alpha^2\gamma^4 \\  -\alpha\beta^2 & 0 & 0 \\ 0 & -\alpha\beta^2 & 0
\end{smallmatrix} \right)\] and that $i_1^2=i_2^2=(i_1 i_2)^3=I_3$. Hence the subgroup $\langle i_1,i_2\rangle$ is isomorphic to $S_3$. In fact, all the elements of $\langle i_1,i_2\rangle$ are 
\begin{alignat*}{2}
i_1 &= \left( \begin{smallmatrix}
    0 & \alpha\gamma^2 & 0 \\  \alpha\beta^2 & 0 & 0 \\ 0 & 0 & -1
\end{smallmatrix}\right), & i_2 &= \left( \begin{smallmatrix}
    -1 & 0 & 0 \\  0 & 0 & \alpha\gamma^2 \\ 0 & \alpha\beta^2 & 0
\end{smallmatrix}\right),\\ 
i_1i_2 &= \left( \begin{smallmatrix}
    0 & 0 & \alpha^2\gamma^4 \\  -\alpha\beta^2 & 0 & 0 \\ 0 & -\alpha\beta^2 & 0
\end{smallmatrix} \right), \hspace{0.5cm} & i_2i_1 &=\left( \begin{smallmatrix}
    0 & -\alpha\gamma^2 & 0 \\  0 & 0 & -\alpha\gamma^2 \\ \alpha^2\beta^4 & 0 & 0\end{smallmatrix} \right),\\ i_1i_2i_1=i_2i_1i_2 &=\left( \begin{smallmatrix}
    0 & 0 & -\alpha^2\gamma^4 \\  0 & -1 & 0 \\ -\alpha^2\beta^4 & 0 & 0\end{smallmatrix} \right), & I_3 &=\left( \begin{smallmatrix}
    1 & 0 & 0 \\  0 & 1 & 0 \\ 0 & 0 & 1\end{smallmatrix} \right).
\end{alignat*}
Clearly, $\langle i_1,i_2\rangle \cap N = 1$. It remains to show that every element in $G$ can be written as a product of an element in $N$ and an element in $\langle i_1,i_2\rangle$. Since $G = \langle N, r, t\rangle$, it suffices to prove this for $t$ and $r$. To factor $t$, one can explicitly compute that $t i_2^{-1} i_1^{-1} \in N$ and hence $t$ can be written as $(ti_2^{-1}i_1^{-1}) (i_1i_2)$. To factor $r$, first note that $r^2\in N$. Since $N$ is normal, we get $tr^{-2}t^{-1}\in N$, hence we can write $r=(tr^{-2}t^{-1})(tr^2t^{-1}r)=(tr^{-2}t^{-1})i_1$.
\end{proof}

Let us illustrate the decomposition with some examples. 

\begin{Exp}\label{Exp: N groups in 3 ways}
    Let $A = \langle \frac{1}{3}(1,2,0) \rangle$. Since conjugation by $T$ permutes the entries of diagonal matrices, we obtain that the normal subgroup $N_1 = A^T$ for the group $G_1 = \langle A, T \rangle$ of type (C) is 
    \[ N_1 = \langle \frac{1}{3}(1,2,0), \frac{1}{3}(2,0,1) \rangle \simeq C_3 \times C_3. \]
    Thus, we have that $G_1 = N_1 \rtimes T$.
    
    Choosing the matrix $ r = \left( \begin{smallmatrix}
        0 & -1 & 0 \\ 1 & 0 & 0 \\ 0 & 0 & 1
    \end{smallmatrix} \right) $ gives rise to a group $G_2 = \langle A, T, r \rangle$ of type (D). This group clearly contains the diagonal matrix $r^2 = \frac{1}{2}(1,1,0)$, and conjugation by $T$ produces the diagonal matrices $\frac{1}{2}(1,0,1)$ and $\frac{1}{2}(0,1,1)$. One can see that all diagonal matrices are generated by $A^T$ and $(r^2)^T$, i.e.\ we have 
    \[ N_2 = \langle A^T, (r^2)^T \rangle \simeq C_6 \times C_6, \]
    and a semi-direct decomposition $ G_2 = N_2 \rtimes S_3$.
\end{Exp}

In particular, we want to point out that different choices of $A$ and $r$ might produce the same groups. This illustrates why our results depend on the group $N$, and not just on $A$. 

\begin{Exp}
    Consider the group $A_1 = \langle \frac{1}{5}(1,1,3) \rangle$, and the matrix $r_1= \left( \begin{smallmatrix}
        0 & -1 & 0 \\ 1 & 0 & 0 \\ 0 & 0 & 1
    \end{smallmatrix} \right) $. The resulting group $G_1 = \langle A_1, T, r_1 \rangle$ has order $600$ and contains the subgroup $N_1$ of diagonal matrices of order $100$, and $N_1 \simeq C_{10} \times C_{10}$. We can see this by noting that $A_1^T$ produces a group $C_5 \times C_5$, and $(r_1^2)^T$ contributes $C_2 \times C_2$. Now, let us construct this group by instead starting with $A_2 = \langle \frac{1}{2}(1,1,0) \rangle$ and choosing    
    $r_2= \left( \begin{smallmatrix}
        0 & \epsilon_5^4 & 0 \\ \epsilon_5^2 & 0 & 0 \\ 0 & 0 & -\epsilon_5^4
    \end{smallmatrix} \right) $ where $\epsilon_5$ is a primitive fifth root of unity. Then $G_2 = \langle A_2 , T, r_2 \rangle$ also is of order $600$. Denote by $N_2$ the subgroup of diagonal matrices. The group $A_2^T$ contributes a $C_2 \times C_2$ to $N_2$, and $r_2^2 = \frac{1}{5}(1,1,3)$. From this, we see that $N_1 \simeq N_2 \simeq C_{10} \times C_{10}$, and it is easy to check that indeed $G_1 = G_2$. 
\end{Exp}

We now fix the decomposition $G = N \rtimes K$, as well as the elements $i_1$ and $i_2$. Note that we gave explicit generators for $K$, so we are justified in writing equality, not just isomorphism. Next, we give a precise description of what groups $N$ can appear. Since we will use this in conjunction with methods from \Cref{Sec: type (A)}, we phrase our results in the corresponding terminology. 

We begin by considering the action of the complement $K$ on $N$, and the induced action on $\hat{N}$. Note that this action is essentially the same as the one that arises from the imprimitivity of the group $G$. 

\begin{Lem}
	Let $G = N \rtimes K \leq \SL_3(\Bbbk)$ be a group of type (C) or (D). Denote by $\rho \colon G \to \SL_3(\Bbbk)$ the embedding, by $\rho_{|N}$ the restriction to $N$, and decompose $\rho_{|N} = \rho_1 \oplus \rho_2 \oplus \rho_3$. The group $K$ acts on $N$, and hence on $\hat{N}$. This induced action permutes the summands $\rho_i$, that is we have $\rho_1^K = \{ \rho_1, \rho_2, \rho_3 \}$.  
\end{Lem}

\begin{proof}
	The decomposition $\rho_{|N} = \rho_1 \oplus \rho_2 \oplus \rho_3$ arises from the three diagonal entries of the matrices in $N$, meaning that $g \in N$ is precisely the diagonal matrix $\diag(\rho_1(g), \rho_2(g), \rho_3(g))$. It is clear that for any diagonal matrix $\diag(a_1, a_2, a_3)$ we have $\diag(a_1, a_2, a_3)^t =t \diag(a_1, a_2, a_3) t^{-1} = \diag( a_3, a_1, a_2)$. Similarly, using the generators $i_1$ and $i_2$ from \Cref{Lem: Semi-direct decomposition}, we see that $\diag(a_1, a_2, a_3)^{i_1} = i_1 \diag(a_1, a_2, a_3) i_1^{-1} = \diag( a_2, a_1, a_3) $, and the corresponding statement holds for $i_2$. Thus, the induced action on $\hat{N}$ is simply the induced permutation action. 
\end{proof}

Next, we extend this action to $\mathbb{Z}^2 $. 

\begin{Rem}
    Recall from \Cref{Sec: type (A)} that we encode an abelian group $N$ and its embedding $\rho \colon N \to \SL_3(\Bbbk)$ as the quotient of the lattice $\mathbb{Z}^2$. More precisely, we have defined the map 
    \[ q \colon \mathbb{Z}^2 \to \hat{N}, e_i \mapsto \rho_i,  \]
    for a fixed basis $\{e_1, e_2 \} $ of $\mathbb{Z}^2$ and $e_3 = -(e_1 + e_2)$, and write $L = \Ker(q)$.  
    We have that $K$ acts on $\hat{N}$ by permuting the $\rho_i$, so we define the corresponding action on $\mathbb{Z}^2$ permuting the $e_i$ accordingly.
\end{Rem}

\begin{Lem}\label{Lem: T-action leaves L invariant}
	Let $G = N \rtimes K$ be of type (C) or (D). Then the sublattice $L$ is invariant under the above defined action of $K$ on $\mathbb{Z}^2$. 
\end{Lem}

\begin{proof}
	Recall that $L$ is the kernel of the morphism 
	\[ q \colon \mathbb{Z}^2 \to \hat{N}, e_i \mapsto \rho_i.\]
	$K$ acts on both $\mathbb{Z}^2$ and on $\hat{N}$ so that $q$ becomes $K$-equivariant, hence the kernel is $K$-invariant. 
\end{proof}

Finally, since we know that $N$ and its representation are encoded in the sublattice $L$, we can describe its embedding. 

\begin{Pro}
\label{pro: matrix shapes}
    Let $G = N \rtimes K$ be of type (C) or (D), and denote by $B$ the matrix for the embedding $ L \hookrightarrow \mathbb{Z}^2$ as in \Cref{Rem: Shape of B matrix}. Then 
    \[ B=\left(\begin{matrix} k_1c & k_2c \\ 0 & c\end{matrix}\right)\] 
    with $0\le k_2<k_1$, $c>0$, and $k_1\mid k_2^2-k_2+1$ for type (C) respectively $k_1\mid \gcd(k_2-2,3)$ for type (D).
\end{Pro}

\begin{proof}
We first deal with type (C), then with type (D).
    Let $B_1=\left(\begin{smallmatrix} a & b \\ 0 & c\end{smallmatrix}\right)$ with $ac=n$ be the matrix defining $L$ such that $\hat{Q}/L = Q_N$. It is clear that $c>0$. Recall that $L$ is invariant under the action of $T$. Since $t(ae_1) = ae_2$ and $t(be_1+ce_2) = be_2+ce_3 = be_2-c(e_1+e_2) = -ce_1+(b-c)e_2$, we get that the matrix $B_2=\left(\begin{smallmatrix} 0 & -c \\ a & b-c\end{smallmatrix}\right)$ also defines $L$. In other words, $B_1^{-1}B_2\in \GL_2(\mathbb{Z})$. Writing out  
    $$
    B_1^{-1}B_2=\frac{1}{ac}\left(\begin{matrix} c & -b \\ 0 & a\end{matrix}\right)\left(\begin{matrix} 0 & -c \\ a & b-c\end{matrix}\right)=\frac{1}{ac}\left(\begin{matrix} -ab & -b^2-c^2+bc \\ a^2 & ab-ac\end{matrix}\right),
    $$
we conclude that $c\mid a$ and $c\mid b$. We write $a=k_1c$, $b=k_2c$. By \Cref{Rem: Shape of B matrix} we have $0\le b<a$, hence $0\le k_2<k_1$. With these substitutions, all the entries of $B_1^{-1}B_2$ are integers except possibly the upper-right entry. We need $ac\mid -b^2-c^2+bc$, in other words, $k_1c^2\mid -k_2^2c^2-c^2+k_2c^2$, that is, $k_1\mid k_2^2-k_2+1$.    

For type (D) all of the above is true as well, hence we can assume $B_1=\left(\begin{smallmatrix} k_1c & k_2c \\ 0 & c\end{smallmatrix}\right)$ with $c>0$, $0\le k_2<k_1$ and $k_1\mid k_2^2-k_2+1$. Additionally, we know that the matrix $B_3=\left(\begin{smallmatrix} 0 & c \\ k_1c & k_2c\end{smallmatrix}\right)$ also defines $L$. In other words, $B_1^{-1}B_3 \in \GL_2(\mathbb{Z})$. We have
\[    B_1^{-1}B_3=\frac{1}{k_1c^2}c\left(\begin{matrix} 1 & -k_2 \\ 0 & k_1\end{matrix}\right)c\left(\begin{matrix} 0 & 1 \\ k_1 & k_2\end{matrix}\right)=\frac{1}{k_1}\left(\begin{matrix} -k_1k_2 & 1-k_2^2 \\ k_1^2 & k_1k_2\end{matrix}\right)\in \GL_2(\mathbb{Z}). \]
All entries of $B_1^{-1}B_3$ are integers except possibly the upper-right entry. Hence we need to impose $k_1\mid k_2^2-1$. Now recall that we already have $k_1\mid k_2^2-k_2+1$. Combining the two, we get 
\[k_1 \mid \gcd(k_2^2-k_2+1,\ k_2^2-1)=\gcd(k_2^2-1,\ k_2-2)=\gcd(k_2-2,3). \qedhere \]
\end{proof}

\begin{Rem}
Let us note some properties we can conclude from the above description of the matrix $B$. 
\begin{enumerate}[wide,  labelindent=0pt]
    \item In \cite{ItoTrihedral}, the author addresses quotient singularities of type (C) and shows that $|N|$ can only be $0$ or $1$ modulo $3$. This is proven by a counting argument. One can also conclude this from \Cref{pro: matrix shapes}. Indeed, $|N|=k_1c^2$. If $3\mid c$, then $|N| \equiv_3 0 $. Otherwise, $c^2 \equiv_3 1 $ and hence $|N|  \equiv_3 k_1 $. We now show that $k_1\not \equiv_3 2$. Since $k_1\mid k_2^2-k_2+1$, it suffices to show that no (positive) divisor of any $k_2^2-k_2+1$ is congruent to $2$ modulo $3$. Further, note that it is enough to show that no \emph{prime} divisor of $k_2^2-k_2+1$ can be of this form. Then if we set $k_2=k+1$, we are left to show that $p \nmid k^2+k+1$ for any $p \equiv_3 2$ and any $k\ge -1$. We fix $p$ and assume that such a $k$ exists. Now we pass to $\mathbb{F}_p$ and, for ease of notation, denote by $\bar{k}$ the corresponding residue class. First, we exclude the case $\bar{k}=1$. Indeed, if $\bar{k}=1$, then $\bar{k}^2+\bar{k}+\bar{1}\neq \bar{0}$. Hence we can assume that $\bar{k}-\bar{1}\in \mathbb{F}_p^{\times}$, and our equation can be rewritten as $(\bar{k}^3-\bar{1})(\bar{k}-\bar{1})^{-1}=0$. Therefore, we need to solve $\bar{k}^3=\bar{1}$. Since $\bar{k}\neq \bar{1}$, we conclude that the order of $\bar{k}$ in $\mathbb{F}_p^{\times}$ is $3$, which is a contradiction since $|\mathbb{F}_p^{\times}|=p-1$ and $3\nmid p-1$ since $p \equiv_3 2$.
    \item In the case of type (D), the matrix $B$ is of the form 
    \[ \left( \begin{matrix}
        c & 0 \\ 0 & c
    \end{matrix} \right) \text{ or } \left( \begin{matrix}
        3c & 2c \\ 0 & c 
    \end{matrix} \right)  \]
    for some $c>0$.
Indeed, from \Cref{pro: matrix shapes}, namely, from $k_1\mid \gcd(k_2-2,3)$ we conclude that either $k_1=1$, in which case $k_2=0$ (since $0\le k_2<k_1$), or $k_1=3$. If $k_1=3$, we get $k_2 \equiv_3 2 $, hence $0\le k_2<k_1$ gives $k_2=2$.

    \end{enumerate}
\end{Rem}

\subsection{The McKay quivers and cuts}\label{SSec: McKay and cuts}
We now describe the McKay quivers, and determine when cuts exist. Again, much of the types (C) and (D) can be treated in parallel, so we only specify the type when necessary. The strategy is to write the McKay quiver $Q_G$ as $Q_N \ast K$, and use the methods from \Cref{Sec: type (A)} for $Q_N$. 

\begin{Lem}
\label{lem: 3 divides}
    Let $G = N \rtimes K$ be of type (C) or (D). If $3  \mid |N|$, then $Q_N$ has a $K$-invariant cut. 
\end{Lem}

\begin{proof}
It suffices to show that for both types (C) and (D) the quiver $Q_N$ has an $S_3$-invariant cut. Recall that $B=\left(\begin{smallmatrix} k_1c & k_2c \\ 0 & c\end{smallmatrix}\right)$ with some extra conditions on $k_1$ and $k_2$ depending on the type. Then $3\mid |N|=k_1c^2$, and if an $S_3$-invariant cut $C$ indeed exists, then necessarily $\theta(C) = \left(\frac{k_1c^2}{3}, \frac{k_1c^2}{3}, \frac{k_1c^2}{3}\right)$, so we first show that $\left(\frac{k_1c^2}{3}, \frac{k_1c^2}{3}, \frac{k_1c^2}{3}\right)$ is indeed a type of some cut of $Q_N$. By \Cref{Theo: SL3 type (A) classification} we know that $\left(\frac{k_1c^2}{3}, \frac{k_1c^2}{3}, \frac{k_1c^2}{3}\right)$ is a type of a cut of $N$ if and only if 
\[
\left(\begin{matrix} \frac{k_1c^2}{3} & \frac{k_1c^2}{3} \end{matrix}\right)
\left(\begin{matrix} k_1c & k_2c \\ 0 & c\end{matrix}\right)=k_1c^2\left(\begin{matrix} \frac{1}{3} & \frac{1}{3} \end{matrix}\right)\left(\begin{matrix} k_1c & k_2c \\ 0 & c\end{matrix}\right)=k_1c^2\left(\begin{matrix}\frac{k_1c}{3} & \frac{c(k_2+1)}{3}\end{matrix}\right)\]
has coordinates divisible by $k_1c^2$, in other words, when $3\mid k_1c$ and $3\mid c(k_2+1)$. If $3\mid c$, we are done. Otherwise, given that $3\mid |N|=k_1c^2$, we conclude that $3\mid k_1$, hence the first divisibility condition holds. Recall from \Cref{pro: matrix shapes} that $k_1$ and $k_2$ satisfy extra conditions. For type (C) we get $3\mid k_1\mid k_2^2-k_2+1$, hence $3\mid k_2^2-k_2+1$, therefore, $3\mid k_2^2-k_2+1+3k_2=(k_2+1)^2$, which implies $3\mid k_2+1$, which is the second divisibility condition we need. The same can be concluded for type (D).
Now that we know that $\gamma = \left(\frac{|N|}{3},\frac{|N|}{3},\frac{|N|}{3}\right)$ is a type of some cut on $Q_N$, we show that the cut $C_\gamma$ from \Cref{Theo: Formula for cut} is $S_3$-invariant. To see this, simply unpack the definition of $\overline{\xi_\gamma}$. For $\gamma = \left(\frac{|N|}{3}, \frac{|N|}{3}, \frac{|N|}{3}\right)$, we have that $$|\Im(\overline{\xi_\gamma})| =  \frac{|N|}{\gcd\left(\frac{|N|}{3}, \frac{|N|}{3}, \frac{|N|}{3}\right)} = 3,$$ and the map $\overline{\xi_\gamma}$ is given by 
\[ \overline{\xi_{\gamma}}(x_1e_1+x_2e_2+L')=(x_1+x_2) + 3\mathbb{Z}. \]
Writing out an arbitrary point in $\mathbb{Z}^2/L$ in a (non-unique) way using all three vectors $e_i$, the expression becomes
\begin{align*}
       \overline{\xi_{\gamma}}((x_1e_1+x_2e_2+x_3e_3)+L')&=\overline{\xi_{\gamma}}(((x_1-x_3)e_1+(x_2-x_3)e_2)+L')\\&=(x_1+x_2-2x_3) + 3\mathbb{Z} =(x_1+x_2+x_3) + 3 \mathbb{Z}.
\end{align*}
Recall that an arrow $a \colon (x_1e_1 + x_2e_2 + x_3 e_3) + L \to (x_1e_1 + x_2e_2 + x_3 e_3 + e_i) + L $ is in $C_\gamma$ if and only if 
\[ \overline{\xi_\gamma}((x_1e_1 + x_2e_2 + x_3 e_3) + L') > \overline{\xi_\gamma}((x_1 e_1 + x_2 e_2 + x_3 e_3 + e_i) + L'),  \]
where the inequality is taken for the smallest nonnegative representatives of the residue classes in $\mathbb{Z}/3 \mathbb{Z}$. This is equivalent to asking whether 
\[ (x_1 + x_2 + x_3) \bmod 3 > (x_1 + x_2 + x_3 + 1) \bmod 3. \]
Clearly, this construction is invariant under every permutation of $e_1,e_2,e_3$.
\end{proof}

\begin{Lem} \label{lem: 3 doesn't divide}
Let $G = N \rtimes K$ be of type (C) or (D), and let $Q_G$ be the McKay quiver of $G$. Then $Q_G$ has a loop if $3  \nmid |N|$.  
\end{Lem}

\begin{proof}
    Recall that $Q_G = Q_N \ast K$. We fix $L \leq \mathbb{Z}^2$ such that $Q_N = \hat{Q}/L$. 
    Assume that $3 \nmid |N|$. This implies $\gcd(3,|N|)=1$, and therefore there exists an integer $k$ such that $3k+1 \equiv_{|N|} 0 $. Let $x_1 = (-k-1)e_1+ke_2 + L \in \mathbb{Z}^2/L$, and consider the following elements:
    \begin{align*}
        x_1&=(-k-1)e_1+ke_2 +L,\\
        x_2&=(-k-1)e_2+ke_3 +L=(-k-1)e_2+k(-e_1-e_2) +L \\&=-ke_1-(2k+1)e_2 + L =-ke_1+ke_2 + L =x_1+e_1 ,\\
        x_3&=(-k-1)e_3+ke_1 +L =(k+1)(e_1+e_2)+ke_1+L\\&=(2k+1)e_1+(k+1)e_2+L=-ke_1+(k+1)e_2+L=x_1+e_1+e_2,\\
        x_1'&=ke_1+(-k-1)e_2+L,\\
        x_2'&=ke_2+(-k-1)e_3+L,\\
        x_3'&=ke_3+(-k-1)e_1+L.
    \end{align*}
    In the case of type (C), the $T$-orbit of $x_1$ is $\{x_1,x_2,x_3\}$. In the case of type (D), the $S_3$-orbit of $x_1$ is $\{x_1,x_2,x_3,x_1',x_2',x_3'\}$. Since $x_2=x_1+e_1$, we can conclude that the quiver $Q_G$ has a loop for both types (C) and (D) if we can show that the $T$-orbit of $x_1$ has size $3$ in the case of type (C), and that the $S_3$-orbit of $x_1$ has size $6$ in the case of type (D). 
    
     In the case of type (C), the $T$-orbit of $x_1$ is of size $3$. Indeed, otherwise it must have size $1$, in which case $x_1=x_2=x_3$, which implies $e_1 +L=e_2+L=0$. This is a contradiction since $N$ contains the abelian group $A$ which we assumed was nontrivial.
    
    In the case of type (D) in order to show that the $S_3$-orbit of $x_1$ is indeed of size $6$, it suffices to show that it contains at least $4$ different elements. Since $S_3$ has a subgroup isomorphic to $T$, the same argument as for type (C) shows that $x_1$, $x_2$ and $x_3$ are pairwise different elements of $N$. Hence it is enough to show that $x_1'$ is different from $x_1$, $x_2$ and $x_3$. 
    
    If we assume $x_1'=x_1$, we obtain $0=x_1-x_1'=(-2k-1)e_1+(2k+1)e_2+L=ke_1-ke_2+L$. Hence $0=3(ke_1-ke_2)+L=3ke_1-3ke_2+L=-e_1+e_2+L$. Because of the $S_3$-action this is equivalent to $0=-e_3+e_1+L=2e_1+e_2+L$. Subtracting these two, we get $3e_1+L=0$. If $e_1+L=0$, we get $x_1=x_2$, which we proved is impossible. Hence $e_1+L$ is an element of order $3$ in $\mathbb{Z}^2/L \simeq N$, which is also impossible since $3\nmid |N|$. 
    
    If we assume $x'_1=x_2$, we get $0=x_2-x_1'=-2ke_1+(2k+1)e_2+L=(k+1)e_1-ke_2+L$, hence $0=3((k+1)e_1-ke_2)+L=(3k+3)e_1-3ke_2+L=2e_1+e_2+L$, which is impossible, as we proved in the case where we assumed $x_1'=x_1$. 
    
    Finally, if we assume $x_1'=x_3$, we get $0=x_1'-x_3=2ke_1+(-2k-2)e_2+L$, hence $0=3(2ke_1+(-2k-2)e_2)+L=6ke_1+(-6k-6)e_2+L=-2e_1-4e_2+L$. Because of the $S_3$-action, this is equivalent to $0=-2e_3-4e_1+L=-2e_1+2e_2+L$. Subtracting these two, we get $6e_2+L=0$. The possible orders of $e_2+L$ in $N$ are thus $1,2,3,6$. Order $1$ is impossible for the same reason as in the previous cases, $3$ and $6$ are also impossible since $3\nmid |N|$. The only possible case is therefore that $e_2+L$ is of order $2$. Because of the $S_3$-action, the same holds for $e_1+L$. Then $N \simeq C_2 \times C_2$, and $|G| = 4 \cdot 6 = 24$. The size of $G$ together with the existence of a faithful irreducible representation in $\SL_3(\Bbbk)$ is enough to determine $G$ uniquely as $G \simeq S_4$. This irreducible faithful representation in $\SL_3(\Bbbk)$ is furthermore unique. We compute the McKay quiver for this group directly in \Cref{Exp: Negative example} and see that both vertices for the two $3$-dimensional irreducible representations have a loop each. 
\end{proof}

Next, we show that the loop we constructed appears to its third power in the potential defining $\Bbbk Q_G/I$. 

\begin{Pro}\label{Pro: coefficient of loop}
    Let $G = N \rtimes K$ be of type (C) or (D) such that $3  \nmid |N|$. Let $Q_G$ be the McKay quiver of $G$, and let $l$ be the loop in $Q_G$ constructed in \Cref{lem: 3 doesn't divide}. Then $l^3$ is a summand in the potential $\omega$ defining $\Bbbk Q_G/I$. 
\end{Pro}

\begin{proof}
    By \cite{BSW}, the potential $\omega$ is a linear combination of cycles of length $3$ in $Q_G$. Following \cite{BSW}, the coefficient of the $3$-cycle $l^3$ is computed as follows. 
    Let $i$ be the vertex in $Q_G$ at which $l$ is supported, and let $\rho \colon G \to \operatorname{SL}_3(\Bbbk)$ be the defining representation of $G$. Without loss of generality, we assume that $\rho(N)$ consists of diagonal matrices. Then $i$ corresponds to an irreducible representation $\chi_i$ of $G$, and $l$ corresponds to a morphism of representations $\chi_i \to \chi_i \otimes \rho$. Applying $l$ three times, we obtain the following sequence 
    \[ \chi_i \xrightarrow[]{l} \chi_i \otimes \rho \xrightarrow[]{l \otimes \operatorname{id}_\rho}  \chi_i \otimes \rho^{\otimes 2} \xrightarrow[]{l \otimes \operatorname{id}_\rho \otimes \operatorname{id}_\rho}  \chi_i \otimes \rho^{\otimes 3} \xrightarrow[]{\operatorname{id}_{\chi_i} \otimes \alpha} \chi_i \otimes \bigwedge^3 \rho,  \]
    where $\alpha \colon \rho^{\otimes 3} \to \bigwedge^3 \rho$ is the antisymmetrizer. Since $\bigwedge^3 \rho$ is the trivial representation, the above sequence gives an automorphism of the irreducible representation $\chi_i$. Hence, this is a scalar multiple $c \operatorname{id}_{\chi_i}$ of the identity morphism, and the coefficient of $l^3$ in $\omega$ is this scalar $c$.  
    We now describe the representation $\chi_i$ explicitly. For convenience, we drop the subscript $i$ and write simply $\chi$. Recall that $Q_G = Q_N \ast K$. We fix $L \leq \mathbb{Z}^2$ such that $Q_N = \hat{Q}/L$. We proved in \Cref{lem: 3 doesn't divide} that there exists a vertex $x$ such that the $T$-orbit of $x$ is $\{x, x + e_1, x + e_1 + e_2\} $. We view these three vertices as linear representations of $N$, and assemble them into a $3$-dimensional representation of $N$. To do this, we denote by $\rho|_N$ the restriction of the defining representation to $N$, and decompose it into $3$ linear characters $\rho|_N = \lambda_1 \oplus \lambda_2 \oplus \lambda_3$. Explicitly, these pick out the three diagonal entries of the matrices in $\rho(N)$. Since the arrows corresponding to $e_i + L$ correspond to multiplication by $\lambda_i$, the vertex $x = a e_1 + b e_2 + L$ corresponds to the representation $\lambda_1^a \cdot \lambda_2^b$. We define first $\chi = ( \lambda_1^a \cdot \lambda_2^b) \oplus (\lambda_1^{a+1} \cdot \lambda_2^b) \oplus (\lambda_1^{a+1} \cdot \lambda_2^{b+1}) $ as a representation of $N$. This extends to a representation of $G$ by mapping $K$ to the usual permutation representation of $T$, respectively, $S_3$. 
    Next, we show that $\chi$ indeed is a vertex supporting a loop, and then compute the coefficient of the loop's third power. 

    For this, we denote by $\{b_1, b_2, b_3\}$ the ordered basis for $\rho$ such that $N$ is diagonalised and $K$ represented by the permutation matrices as before. Similarly, we denote an ordered basis for $\chi$ with the same properties by $\{ e_1, e_2, e_3\}$. Then the elementary tensors $\{e_i \otimes b_j \mid 1 \leq i, j \leq 3 \}$ form a basis for $\chi \otimes \rho$. Fix the vectors 
    \[ f_1 = e_3 \otimes b_3, \quad f_2 = e_1 \otimes b_1, \quad f_3 = e_2 \otimes b_2. \]
    We claim that $\chi \simeq \langle f_1, f_2, f_3 \rangle$. It is easy to check that $K$ acts on $\langle f_1, f_2, f_3 \rangle $ in the same permutation representation as for $\chi$ and $\rho$. To see that $N$ acts accordingly, we apply $n \in N$ to the vectors explicitly. We have 
    \begin{align*}
        n \cdot (f_1) &= \lambda_1(n)^{a+1} \cdot \lambda_2(n)^{b+1} e_3 \otimes \lambda_3(n) b_3 \\
        &=\lambda_1(n)^{a+1} \cdot \lambda_2^{b+1}(n) \cdot \lambda_1(n)^{-1} \cdot \lambda_2(n)^{-1} e_3 \otimes  b_3 \\
        &= \lambda_1(n)^{a} \cdot \lambda_2^{b}(n) e_3 \otimes b_3,
    \end{align*}
    and similarly we have 
    \begin{align*}
        n \cdot (f_2) &=  \lambda_1(n)^{a} \cdot \lambda_2(n)^{b} e_1 \otimes \lambda_1(n) b_1 \\
        &= \lambda_1(n)^{a+1} \cdot \lambda_2(n)^{b} e_1 \otimes b_1
    \end{align*}
    and 
    \begin{align*}
        n \cdot (f_3) &=  \lambda_1(n)^{a+1} \cdot \lambda_2(n)^{b} e_2 \otimes \lambda_2(n) b_2 \\
        &= \lambda_1(n)^{a+1} \cdot \lambda_2(n)^{b+1} e_2 \otimes b_2.
    \end{align*}
    Thus, we have that the map $e_i \mapsto f_i$ gives a non-trivial homomorphism $\chi \to \chi \otimes \rho$, and we choose this homomorphism for the loop $l$. Note that this is the same computation one performs to cover the special case from \Cref{lem: 3 doesn't divide} in \Cref{Exp: Negative example}. Now we compute the coefficient of $l^3$. With the bases established, this step is easy, since we have the following sequence
    \[  e_i \mapsto f_i = (e_{i-1} \otimes b_{i-1}) \mapsto (e_{i-2} \otimes  b_{i-2} \otimes b_{i-1}) \mapsto (e_{i-3} \otimes b_{i-3} \otimes b_{i-2} \otimes b_{i-1}), \]
    where the indices are taken from $\{1, 2, 3\} $ modulo $3$. The antisymmetrizer takes $(b_{i-3} \otimes b_{i-2} \otimes b_{i-1})$ to $1$, and hence the scalar of $l^3$ in $\omega$ is non-zero. 
\end{proof}

Combining the previous results, we arrive to the following theorem.

\begin{Theo}
\label{classification}
Let $G = N \rtimes K$ be of type (C) or (D). Then $R \ast G \simeq_M \Bbbk Q_G/I$ admits a $3$-preprojective cut if and only if $3 \mid |N|$.
\end{Theo}

\begin{proof}
    If $3 \mid |N|$, it follows from \Cref{lem: 3 divides} that $Q_N$ has a $K$-invariant cut. It follows from \Cref{Cor: Skewing by graded auts} that $Q_G = Q_N \ast K$ then also has a higher preprojective cut. 
    Conversely, if $3 \nmid |N|$, it follows from \Cref{lem: 3 doesn't divide} that $Q_G$ has a loop which by \Cref{Pro: coefficient of loop} appears to the power $3$ in the potential, so it follows from \Cref{Pro: loops in support} that there is no higher preprojective cut on $\Bbbk Q_G/I$. 
\end{proof}

We illustrate our results with the groups from \Cref{Exp: N groups in 3 ways}. Parts of the first example already appeared in \cite{DramburgGasanova}. In the following examples, many of the $2$-representation infinite algebras we obtain by cutting the McKay quiver are so-called \emph{levelled algebras}. We remind the reader that this is equivalent to saying that the cut quiver is the Hasse diagram of a \emph{ranked poset}. Furthermore, if the length of a longest path in the cut quiver is $n$, the levelled algebra is called $n$-levelled. We make use of the levelled structure when drawing the quivers, placing the vertices in the same level into a column and drawing arrows between columns. 

\begin{Exp} \label{Exp: Positive case for (C)}
    Let $A = \langle \frac{1}{3}(1,2,0) \rangle$. We saw in \Cref{Exp: N groups in 3 ways} that $G = \langle A, T \rangle$ of type (C) is $(C_3 \times C_3) \rtimes T$. Thus, we know from \Cref{lem: 3 divides} that there exists a $T$-invariant cut on the quiver $Q_N$, where $N = A^T \simeq C_3 \times C_3$. Let us describe the quiver and the cut, and from this the quiver for $G$ and the skewed cut. To draw the quiver $Q_N$, we follow the convention from \Cref{Sec: type (A)}, and draw a portion of the infinite quiver $\hat{Q}$ together with the images of the vertices under $q$. We denote the elements of $\hat{N}$ by $\{ (i,j) \mid 0 \leq i,j < 3\}$. The action of $T$ corresponds to rotation around the vertex $(0,0)$. We construct the $T$-invariant cut as in \Cref{Sec: type (A)} and draw the cut arrows in red.

\[ 
    \begin{tikzpicture}[-Stealth]
        \begin{scope}[rotate=-15.4,inner sep=1.5mm]
        \begin{scriptsize}

        \coordinate (O) at (0, 0, 0);
        \coordinate (X) at (1,0,-1);
        \coordinate (Y) at (0, -1, 1);
        \coordinate (Z) at (-1, 1, 0);

        \node (A0) at (O) {(0,0)};
        \node (A6) at ($1*(X)$) {(1,2)};
        \node (A3) at ($2 *(X)$) {(2,1)};
        \node (A5) at ($(X)+(Y)$) {(0,2)};
        \node (A4) at ($2*(X)+(Y)$) {(1,1)};
        \node (A1) at ($-1*(Y)$) {(1,0)};
        \node (A2) at ($(X)-1*(Y)$) {(2,2)};
        \draw[red] (A1)--(A0);
        \draw[red] (A1)--(A2);
        \draw[red] (A3)--(A2);
        \draw[red] (A3)--(A4);
        \draw[red] (A5)--(A4);
        \draw[red] (A5)--(A0);
        \draw (A0)--(A6);
        \draw (A6)--(A1);
        \draw (A2)--(A6);
        \draw (A6)--(A3);
        \draw (A4)--(A6);
        \draw (A6)--(A5);

        \node(B0) at ($-1*(Y)-2*(X)$) {(2,2)}; 
        \node (B6) at ($-1*(Y)-1*(X)$) {(0,1)};
        \node (B3) at ($-1*(Y)$) {(1,0)};
        \node (B5) at ($-1*(X)$) {(2,1)};
        \node (B4) at ($(O)$) {(0,0)};
        \node (B1) at ($-2*(Y)-2*(X)$) {(0,2)};
        \node (B2) at ($-2*(Y)-1*(X)$) {(1,1)};
        \draw[red] (B1)--(B0);
        \draw[red] (B1)--(B2);
        \draw[red] (B3)--(B2);
        \draw[red] (B3)--(B4);
        \draw[red] (B5)--(B4);
        \draw[red] (B5)--(B0);
        \draw (B0)--(B6);
        \draw (B6)--(B1);
        \draw (B2)--(B6);
        \draw (B6)--(B3);
        \draw (B4)--(B6);
        \draw (B6)--(B5);

        \node (C0) at ($-1*(X)+1*(Y)$) {(1,1)};
        \node (C6) at ($1*(Y)$) {(2,0)};
        \node (C3) at ($1*(X)+1*(Y)$) {(0,2)};
        \node (C5) at ($2*(Y)$) {(1,0)};
        \node (C4) at ($1*(X)+2*(Y)$) {(2,2)};
        \node (C1) at ($-1*(X)$) {(2,1)};
        \node (C2) at ($(O)$) {(0,0)};
        \draw[red] (C1)--(C0);
        \draw[red] (C1)--(C2);
        \draw[red] (C3)--(C2);
        \draw[red] (C3)--(C4);
        \draw[red] (C5)--(C4);
        \draw[red] (C5)--(C0);
        \draw (C0)--(C6);
        \draw (C6)--(C1);
        \draw (C2)--(C6);
        \draw (C6)--(C3);
        \draw (C4)--(C6);
        \draw (C6)--(C5);
        
     \end{scriptsize}
        \end{scope}
 \end{tikzpicture}
    \] 
    Drawing and rearranging the quiver for the resulting $2$-representation infinite algebra $\Lambda_2$ produces the following $2$-levelled quiver. 
    \[ 
    \begin{tikzcd}
{(0,0)} \arrow[rr] \arrow[rrd] \arrow[rrdd] &  & {(1,2)} \arrow[rr] \arrow[rrd] \arrow[rrdd] &  & {(2,1)} \\
{(1,1)} \arrow[rru] \arrow[rr] \arrow[rrd]  &  & {(2,0)} \arrow[rru] \arrow[rr] \arrow[rrd]  &  & {(0,2)} \\
{(2,2)} \arrow[rruu] \arrow[rru] \arrow[rr] &  & {(0,1)} \arrow[rruu] \arrow[rru] \arrow[rr] &  & {(1,0)}
\end{tikzcd}
    \]
    Now we consider the action of $T$ on $Q_N$ to produce the corresponding quiver and cut for $G$. One can check, using the computer or \Cref{Theo: Demonet}, that the McKay quiver for $G$ is the one given below. To understand its structure, note that the orbits of the $T$-action on $Q_N$ are easy to read off. The vertices $(0,0)$, $(1,1)$ and $(2,2)$ are stabilised, and there are $2$ orbits of size $3$, given by $o_1 = \{(0,2), (2,1), (1,0)  \}$ and $o_2 = \{ (1,2), (2,0), (0,1) \}$. Recall that by \Cref{Def: skewed} every $G$-orbit of $Q_0$ with a representative $e$ gives rise to  $|\Irr(\Stab_G(e))|$ elements in $(Q\ast G)_0$. The non-trivial orbits therefore give rise to a single vertex each, while the stabilised vertices split into $3$ vertices each, denoted by $x_i$, $y_i$ and $z_i$. The vertices $x_i$, $y_i$ and $z_i$ are drawn twice for convenience and identified, and we draw the skewed cut in red.   
    \[ 
    \begin{tikzcd}[row sep= tiny]
x_1 \arrow[rrddddd] &  &                                                                    &  &                                                                                                                                  &  & x_1 \\
x_2 \arrow[rrdddd]  &  &                                                                    &  &                                                                                                                                  &  & x_2 \\
x_3 \arrow[rrddd]   &  &                                                                    &  &                                                                                                                                  &  & x_3 \\
                    &  &                                                                    &  &                                                                                                                                  &  &     \\
y_1 \arrow[rrd]     &  &                                                                    &  &                                                                                                                                  &  & y_1 \\
y_2 \arrow[rr]      &  & o_2  \arrow[rr, shift left=2] \arrow[rr, shift right=2] \arrow[rr] &  & o_1 \arrow[color = red,rruuuuu] \arrow[color = red,rruuuu] \arrow[color = red,rruuu] \arrow[color = red,rru] \arrow[color = red,rr] \arrow[color = red,rrd] \arrow[color = red,rrddd] \arrow[color = red,rrdddd] \arrow[color = red,rrddddd] &  & y_2 \\
y_3 \arrow[rru]     &  &                                                                    &  &                                                                                                                                  &  & y_3 \\
                    &  &                                                                    &  &                                                                                                                                  &  &     \\
z_1 \arrow[rruuu]   &  &                                                                    &  &                                                                                                                                  &  & z_1 \\
z_2 \arrow[rruuuu]  &  &                                                                    &  &                                                                                                                                  &  & z_2 \\
z_3 \arrow[rruuuuu] &  &                                                                    &  &                                                                                                                                  &  & z_3
\end{tikzcd}
    \]
    The resulting $2$-representation infinite algebra $\Lambda_3$ is again $2$-levelled. It can also be obtained directly by taking the quiver for the previous $2$-representation infinite algebra and noting that $T$ stabilises the points in the left column while permuting the points in the middle column as well as the right column. 
    
    For completeness, we also point out that the group $N$ contains $N' = \langle \frac{1}{3}(1,1,1) \rangle $. The skew-group algebra $R \ast N'$ is the $3$-preprojective alegbra of the $2$-Beilinson algebra $\Lambda_1$. More precisely, the McKay quiver for $N'$ with a cut is given by the following.   
    \[\begin{tikzcd}
                                                                & 1 \arrow[rd, shift left=2] \arrow[rd, shift right=2] \arrow[rd] &                                                                 \\
0 \arrow[ru, shift left=2] \arrow[ru, shift right=2] \arrow[ru] &                                                                 & 2 \arrow[color = red,ll, shift left=2] \arrow[color = red,ll, shift right=2] \arrow[color = red,ll]
\end{tikzcd} \]
    Passing from this quiver to the one for $N$ can also be done by skewing by another $C_3$-action which stabilises the vertices and permutes the arrows between pairs of vertices. Thus, we have a chain of three $2$-representation infinite algebras obtained by iterated skewing
    \[ (\Lambda_1 \ast C_3) \ast T \simeq \Lambda_2 \ast T \simeq \Lambda_3.  \]

    Next, we add the matrix $r = \left( \begin{smallmatrix}
        0& -1 & 0 \\-1 & 0 & 0 \\ 0 & 0 & -1
    \end{smallmatrix} \right)$ to generate a group $G_2 = \langle A, T, r \rangle \supseteq G$ of type (D). Since $r^2 = I_3$ is the identity, it is easy to see that $G_2 \simeq N \rtimes S_3 \simeq (C_3 \times C_3) \rtimes S_3$. We draw the skewed quiver and cut below, which again can be computed directly from $G_2$ or using \Cref{Theo: Demonet}. To understand the structure of the quiver $Q_{G_2}$, we first note that the same $T$-orbits on $Q_N$ arise as before. The extra action of $r$ then exchanges $(1,1)$ and $(2,2)$, while stabilising the remaining orbits. Thus, we now obtain $3 = |\operatorname{Irr}(S_3)|$ vertices for the stabilised vertex $(0,0)$, and $2$ vertices for the two orbits of size $3$, and $3$ vertices for the orbit $o_3 = \{ (1,1), (2,2) \}$. 
    \[
    \begin{tikzcd}[row sep= tiny]
1 \arrow[rrddd]            &  &                                                               &  &                                                                  &  & 1 \\
                           &  &                                                               &  &                                                                  &  &   \\
3 \arrow[rrd] \arrow[rrdd] &  &                                                               &  &                                                                  &  & 3 \\
4 \arrow[rr] \arrow[rrd]   &  & 7 \arrow[rrd, shift left] \arrow[rrd, shift right] \arrow[rr] &  & 10 \arrow[color = red,rruuu] \arrow[color = red,rru] \arrow[color = red,rr] \arrow[color = red,rrd] \arrow[color = red,rrdd] &  & 4 \\
5 \arrow[rru] \arrow[rr]   &  & 8 \arrow[rr] \arrow[rru, shift left] \arrow[rru, shift right] &  & 9 \arrow[color = red,rruu] \arrow[color = red,rru] \arrow[color = red,rr] \arrow[color = red,rrd] \arrow[color = red,rrddd]  &  & 5 \\
6 \arrow[rruu] \arrow[rru] &  &                                                               &  &                                                                  &  & 6 \\
                           &  &                                                               &  &                                                                  &  &   \\
2 \arrow[rruuu]            &  &                                                               &  &                                                                  &  & 2
\end{tikzcd}
    \]
    Since $S_3 \simeq C_3 \rtimes C_2$ is solvable, we can also obtain this quiver from $Q_G$ by acting with a group of order $2$ appropriately. More precisely, the group of order two stabilises the vertices $o_1$ and $o_2$ in $Q_G$. The only other stabilised vertex is the one corresponding to the trivial representation of $G$, which in the quiver $Q_{G_2}$ gives rise to the vertices we labeled $1$ and $2$. This also makes it clear that the vertices labelled $1$ and $2$ correspond to the two $1$-dimensional representations of $G_2$, while vertices $3,4,5,6$ correspond to the $2$-dimensional representations of $G_2$, and the vertices $7,8,9,10$ correspond to the $3$-dimensional representations of $G_2$. The action of $C_2$ on arrows involves some scaling, so we omit the details. We have obtained a $2$-representation infinite algebra $\Lambda_4$ which fits into the chain of iterated skews. 
    \[ ((\Lambda_1 \ast C_3) \ast T)\ast C_2 \simeq (\Lambda_2 \ast T) \ast C_2 \simeq \Lambda_3 \ast C_2 \simeq \Lambda_4.  \]
\end{Exp}

While the above computation can reasonably be done by hand, we now continue with a larger example that was aided by computations in GAP \cite{GAP4}. In particular, when we simply label the vertices of a McKay quiver by consecutive integers, these are the labels given to the irreducible characters by GAP. 

\begin{Exp}
    We continue the example above, but create a different group of type (D). Let $A = \langle \frac{1}{3}(1,2,0) \rangle$, and recall from \Cref{Exp: N groups in 3 ways} that $G = \langle A, T, r \rangle$ for $ r = \left( \begin{smallmatrix}
        0 & -1 & 0 \\ 1 & 0 & 0 \\ 0 & 0 & 1
    \end{smallmatrix} \right) $ is a group of type (D). The group of diagonal matrices $N \trianglelefteq G$ is now larger than just $C_3 \times C_3$. Indeed, it is $N = \langle \frac{1}{6}(5, 1, 0) , \frac{1}{6}( 1, 0, 5)  \rangle  \simeq C_6 \times C_6$. As before, using \Cref{lem: 3 divides}, we see that the quiver $Q_N$ for $N$ has a cut that is invariant under the action of $S_3$, which we draw below. The action of $S_3$ corresponds to rotation around the vertex $(0,0)$ and reflection along the axes given by the three outgoing arrows of the vertex $(0,0)$.  
\[ 
    \begin{tikzpicture}[-Stealth]
        \begin{scope}[rotate=-15.8,inner sep=1.5mm]
        \begin{scriptsize}

        \coordinate (O) at (0, 0, 0);
        \coordinate (X) at (1,0,-1);
        \coordinate (Y) at (0, -1, 1);
        \coordinate (Z) at (-1, 1, 0);

        \node (A0) at (O) {(0,0)};
        \node (A6) at ($1*(X)$) {(5,1)};
        \node (A3) at ($2 *(X)$) {(4,2)};
        \node (A5) at ($(X)+(Y)$) {(0,1)};
        \node (A4) at ($2*(X)+(Y)$) {(5,2)};
        \node (A1) at ($-1*(Y)$) {(5,0)};
        \node (A2) at ($(X)-1*(Y)$) {(4,1)};
        \draw[red] (A1)--(A0);
        \draw[red] (A1)--(A2);
        \draw[red] (A3)--(A2);
        \draw[red] (A3)--(A4);
        \draw[red] (A5)--(A4);
        \draw[red] (A5)--(A0);
        \draw (A0)--(A6);
        \draw (A6)--(A1);
        \draw (A2)--(A6);
        \draw (A6)--(A3);
        \draw (A4)--(A6);
        \draw (A6)--(A5);

        \node(B0) at ($-1*(Y)-2*(X)$) {(1,4)}; 
        \node (B6) at ($-1*(Y)-1*(X)$) {(0,5)};
        \node (B3) at ($-1*(Y)$) {(5,0)};
        \node (B5) at ($-1*(X)$) {(1,5)};
        \node (B4) at ($(O)$) {(0,0)};
        \node (B1) at ($-2*(Y)-2*(X)$) {(0,4)};
        \node (B2) at ($-2*(Y)-1*(X)$) {(5,5)};
        \draw[red] (B1)--(B0);
        \draw[red] (B1)--(B2);
        \draw[red] (B3)--(B2);
        \draw[red] (B3)--(B4);
        \draw[red] (B5)--(B4);
        \draw[red] (B5)--(B0);
        \draw (B0)--(B6);
        \draw (B6)--(B1);
        \draw (B2)--(B6);
        \draw (B6)--(B3);
        \draw (B4)--(B6);
        \draw (B6)--(B5);

        \node (C0) at ($-1*(X)+1*(Y)$) {(2,5)};
        \node (C6) at ($1*(Y)$) {(1,0)};
        \node (C3) at ($1*(X)+1*(Y)$) {(0,1)};
        \node (C5) at ($2*(Y)$) {(2,0)};
        \node (C4) at ($1*(X)+2*(Y)$) {(1,1)};
        \node (C1) at ($-1*(X)$) {(1,5)};
        \node (C2) at ($(O)$) {(0,0)};
        \draw[red] (C1)--(C0);
        \draw[red] (C1)--(C2);
        \draw[red] (C3)--(C2);
        \draw[red] (C3)--(C4);
        \draw[red] (C5)--(C4);
        \draw[red] (C5)--(C0);
        \draw (C0)--(C6);
        \draw (C6)--(C1);
        \draw (C2)--(C6);
        \draw (C6)--(C3);
        \draw (C4)--(C6);
        \draw (C6)--(C5);

        \node (D0) at ($-3*(X)$) {(3,3)};
        \node (D6) at ($-2*(X)$) {(2,4)};
        \node (D3) at ($-1*(X)$) {(1,5)};
        \node (D5) at ($-2*(X)+1*(Y)$) {(3,4)};
        \node (D4) at ($-1*(X)+1*(Y)$) {(2,5)};
        \node (D1) at ($-3*(X)-(Y)$) {(2,3)};
        \node (D2) at ($-2*(X)-(Y)$) {(1,4)};
        \draw[red] (D1)--(D0);
        \draw[red] (D1)--(D2);
        \draw[red] (D3)--(D2);
        \draw[red] (D3)--(D4);
        \draw[red] (D5)--(D4);
        \draw[red] (D5)--(D0);
        \draw (D0)--(D6);
        \draw (D6)--(D1);
        \draw (D2)--(D6);
        \draw (D6)--(D3);
        \draw (D4)--(D6);
        \draw (D6)--(D5);

        \node(E0) at ($2*(Y)+1*(X)$) {(1,1)}; 
        \node (E6) at ($2*(Y)+2*(X)$) {(0,2)};
        \node (E3) at ($3*(X)+2*(Y)$) {(5,3)};
        \node (E5) at ($2*(X)+3*(Y)$) {(1,2)};
        \node (E4) at ($3*(X)+3*(Y)$) {(0,3)};
        \node (E1) at ($(Y)+(X)$) {(0,1)};
        \node (E2) at ($(Y)+2*(X)$) {(5,2)};
        \draw[red] (E1)--(E0);
        \draw[red] (E1)--(E2);
        \draw[red] (E3)--(E2);
        \draw[red] (E3)--(E4);
        \draw[red] (E5)--(E4);
        \draw[red] (E5)--(E0);
        \draw (E0)--(E6);
        \draw (E6)--(E1);
        \draw (E2)--(E6);
        \draw (E6)--(E3);
        \draw (E4)--(E6);
        \draw (E6)--(E5);

        \node (F0) at ($-1*(X)-2*(Y)$) {(5,5)};
        \node (F6) at ($-2*(Y)$) {(4,0)};
        \node (F3) at ($1*(X)-2*(Y)$) {(3,1)};
        \node (F5) at ($-1*(Y)$) {(5,0)};
        \node (F4) at ($1*(X)-1*(Y)$) {(4,1)};
        \node (F1) at ($-1*(X)-3*(Y)$) {(4,5)};
        \node (F2) at ($-3*(Y)$) {(3,0)};
        \draw[red] (F1)--(F0);
        \draw[red] (F1)--(F2);
        \draw[red] (F3)--(F2);
        \draw[red] (F3)--(F4);
        \draw[red] (F5)--(F4);
        \draw[red] (F5)--(F0);
        \draw (F0)--(F6);
        \draw (F6)--(F1);
        \draw (F2)--(F6);
        \draw (F6)--(F3);
        \draw (F4)--(F6);
        \draw (F6)--(F5);

        \node (G0) at ($-3*(X)-3*(Y)$) {(0,3)};
        \node (G6) at ($-2*(X)-3*(Y)$) {(5,4)};
        \node (G3) at ($-1*(X)-3*(Y)$) {(4,5)};
        \node (G5) at ($-2*(X)-2*(Y)$) {(0,4)};
        \node (G4) at ($-1*(X)-2*(Y)$) {(5,5)};
        \node (G1) at ($-4*(Y)-3*(X)$) {(5,3)};
        \node (G2) at ($-2*(X)-4*(Y)$) {(4,4)};
        \draw[red] (G1)--(G0);
        \draw[red] (G1)--(G2);
        \draw[red] (G3)--(G2);
        \draw[red] (G3)--(G4);
        \draw[red] (G5)--(G4);
        \draw[red] (G5)--(G0);
        \draw (G0)--(G6);
        \draw (G6)--(G1);
        \draw (G2)--(G6);
        \draw (G6)--(G3);
        \draw (G4)--(G6);
        \draw (G6)--(G5);

        \node (H0) at ($3*(Y)$) {(3,0)};
        \node (H6) at ($1*(X)+3*(Y)$) {(2,1)};
        \node (H3) at ($2*(X)+3*(Y)$) {(1,2)};
        \node (H5) at ($(X)+4*(Y)$) {(3,1)};
        \node (H4) at ($2*(X)+4*(Y)$) {(2,2)};
        \node (H1) at ($2*(Y)$) {(2,0)};
        \node (H2) at ($(X)+2*(Y)$) {(1,1)};
        \draw[red] (H1)--(H0);
        \draw[red] (H1)--(H2);
        \draw[red] (H3)--(H2);
        \draw[red] (H3)--(H4);
        \draw[red] (H5)--(H4);
        \draw[red] (H5)--(H0);
        \draw (H0)--(H6);
        \draw (H6)--(H1);
        \draw (H2)--(H6);
        \draw (H6)--(H3);
        \draw (H4)--(H6);
        \draw (H6)--(H5);

         \node(I0) at ($-1*(Y)+1*(X)$) {(4,1)}; 
        \node (I6) at ($-1*(Y)+2*(X)$) {(3,2)};
        \node (I3) at ($-1*(Y)+3*(X)$) {(2,3)};
        \node (I5) at ($2*(X)$) {(4,2)};
        \node (I4) at ($3*(X)$) {(3,3)};
        \node (I1) at ($-2*(Y)+1*(X)$) {(3,1)};
        \node (I2) at ($-2*(Y)+2*(X)$) {(2,2)};
        \draw[red] (I1)--(I0);
        \draw[red] (I1)--(I2);
        \draw[red] (I3)--(I2);
        \draw[red] (I3)--(I4);
        \draw[red] (I5)--(I4);
        \draw[red] (I5)--(I0);
        \draw (I0)--(I6);
        \draw (I6)--(I1);
        \draw (I2)--(I6);
        \draw (I6)--(I3);
        \draw (I4)--(I6);
        \draw (I6)--(I5);

         \node(J0) at ($2*(Y)-2*(X)$) {(4,4)}; 
        \node (J6) at ($2*(Y)-1*(X)$) {(3,5)};
        \node (J3) at ($2*(Y)$) {(2,0)};
        \node (J5) at ($-1*(X)+3*(Y)$) {(4,5)};
        \node (J4) at ($3*(Y)$) {(3,0)};
        \node (J1) at ($1*(Y)-2*(X)$) {(3,4)};
        \node (J2) at ($1*(Y)-1*(X)$) {(2,5)};
        \draw[red] (J1)--(J0);
        \draw[red] (J1)--(J2);
        \draw[red] (J3)--(J2);
        \draw[red] (J3)--(J4);
        \draw[red] (J5)--(J4);
        \draw[red] (J5)--(J0);
        \draw (J0)--(J6);
        \draw (J6)--(J1);
        \draw (J2)--(J6);
        \draw (J6)--(J3);
        \draw (J4)--(J6);
        \draw (J6)--(J5);

         \node (K0) at ($2*(X)+1*(Y)$) {(5,2)};
        \node (K6) at ($3*(X)+1*(Y)$) {(4,3)};
        \node (K3) at ($4*(X)+1*(Y)$) {(3,4)};
        \node (K5) at ($3*(X)+2*(Y)$) {(5,3)};
        \node (K4) at ($4*(X)+2*(Y)$) {(4,4)};
        \node (K1) at ($2*(X)$) {(4,2)};
        \node (K2) at ($3*(X)$) {(3,3)};
        \draw[red] (K1)--(K0);
        \draw[red] (K1)--(K2);
        \draw[red] (K3)--(K2);
        \draw[red] (K3)--(K4);
        \draw[red] (K5)--(K4);
        \draw[red] (K5)--(K0);
        \draw (K0)--(K6);
        \draw (K6)--(K1);
        \draw (K2)--(K6);
        \draw (K6)--(K3);
        \draw (K4)--(K6);
        \draw (K6)--(K5);

        \node (L0) at ($-4*(X)-2*(Y)$) {(2,2)};
        \node (L6) at ($-2*(Y)-3*(X)$) {(1,3)};
        \node (L3) at ($-2*(X)-2*(Y)$) {(0,4)};
        \node (L5) at ($-1*(Y)-3*(X)$) {(2,3)};
        \node (L4) at ($-2*(X)-1*(Y)$) {(1,4)};
        \node (L1) at ($-4*(X)-3*(Y)$) {(1,2)};
        \node (L2) at ($-3*(X)-3*(Y)$) {(0,3)};
        \draw[red] (L1)--(L0);
        \draw[red] (L1)--(L2);
        \draw[red] (L3)--(L2);
        \draw[red] (L3)--(L4);
        \draw[red] (L5)--(L4);
        \draw[red] (L5)--(L0);
        \draw (L0)--(L6);
        \draw (L6)--(L1);
        \draw (L2)--(L6);
        \draw (L6)--(L3);
        \draw (L4)--(L6);
        \draw (L6)--(L5);

     \end{scriptsize}
        \end{scope}
 \end{tikzpicture}
    \] 
  We first draw the quiver for the subgroup $H \leq G$ of type (C), which we compute using GAP. Several vertices have been drawn repeatedly to make the structure of the quiver more apparent, but of course vertices with the same label are identified. Furthermore, we identify arrows when at least one of the involved vertices has been drawn multiple times. The skewed cut is drawn in red. 
  \[\begin{tikzcd}[row sep= tiny, /tikz/commutative diagrams/arrow style= tikz]
14 \arrow[rrd] \arrow[rr]                                    &  & 12 \arrow[rrd]                                                            &  & 11 \arrow[color = red,rr] &  & 14                                                \\
\begin{smallmatrix} 1 \\ 2 \\ 3 \end{smallmatrix} \arrow[rr, start anchor = {[yshift=1.5ex]}] \arrow[rr] \arrow[rr, start anchor = {[yshift=-1.5ex]}] &  & 18 \arrow[rr, shift left] \arrow[rr, shift right] \arrow[rrd] \arrow[rru] &  & 17 \arrow[color = red, rru] \arrow[color = red, rrd]  \arrow[color = red,rr, end anchor = {[yshift=1.5ex]}] \arrow[color = red,rr] \arrow[color = red,rr, end anchor = {[yshift=-1.5ex]}] &  & \begin{smallmatrix} 1 \\ 2 \\ 3 \end{smallmatrix} \\
13 \arrow[rrd] \arrow[rru] \arrow[rr]                        &  & 12 \arrow[rrd] \arrow[rru]                                                &  & 11 \arrow[color = red,rr] &  & 13                                                \\
\begin{smallmatrix} 4 \\ 6 \\ 8 \end{smallmatrix} \arrow[rr, start anchor = {[yshift=1.5ex]}] \arrow[rr] \arrow[rr, start anchor = {[yshift=-1.5ex]}] &  & 16 \arrow[rr, shift left] \arrow[rr, shift right] \arrow[rru] \arrow[rrd] &  & 19 \arrow[color = red, rru] \arrow[color = red, rrd] \arrow[color = red,rr, end anchor = {[yshift=1.5ex]}] \arrow[color = red,rr] \arrow[color = red,rr, end anchor = {[yshift=-1.5ex]}] &  & \begin{smallmatrix} 4 \\ 6 \\ 8 \end{smallmatrix} \\
10 \arrow[rrd] \arrow[rru] \arrow[rr]                        &  & 12 \arrow[rrd] \arrow[rru]                                                &  & 11 \arrow[color = red,rr] &  & 10                                                \\
\begin{smallmatrix} 5 \\ 7 \\ 9 \end{smallmatrix} \arrow[rr, start anchor = {[yshift=1.5ex]}] \arrow[rr] \arrow[rr, start anchor = {[yshift=-1.5ex]}] &  & 20 \arrow[rr, shift left] \arrow[rr, shift right] \arrow[rru] \arrow[rrd] &  & 15 \arrow[color = red, rru] \arrow[color = red, rrd] \arrow[color = red,rr, end anchor = {[yshift=1.5ex]}] \arrow[color = red,rr] \arrow[color = red,rr, end anchor = {[yshift=-1.5ex]}] &  & \begin{smallmatrix} 5 \\ 7 \\ 9 \end{smallmatrix} \\
14 \arrow[rru] \arrow[rr]                                    &  & 12 \arrow[rru]                                                            &  & 11 \arrow[color = red,rr] &  & 14                                               
\end{tikzcd}
  \]
  To compute the quiver for the group of type (D), we now need to skew by the action of the group of order $2$. To compute this action, we make use of GAP by first finding a complement for $H$ in $G$, letting it act by conjugation on $H$ and then computing the induced action on the irreducible characters of $H$. This action permutes the vertices by 
  \[ \sigma = (2,3)(4,5)(6,7)(8,9)(13,14)(15,19)(16,20).\]
  The resulting quiver is drawn below. Again, some vertices are drawn repeatedly for ease of reading, and the same identifications are made as before. The skewed cut is drawn in red.  
  \[\begin{tikzcd}[row sep= tiny, /tikz/commutative diagrams/arrow style= tikz]
                                                                             & 3 \arrow[rrd]                                                   &  &                                                                             &  & 15 \arrow[color = red,rr]                                                &  & 3  &                                                   \\
                                                                         & 2 \arrow[rr]                                                    &  & 14 \arrow[rru] \arrow[rr] \arrow[rrd]                                       &  & 9 \arrow[color = red,rrd] \arrow[color = red,rrdd]                                   &  & 2  &                                                   \\
                                                                         & 7 \arrow[rr] \arrow[rrd]                                        &  & 10 \arrow[rr] \arrow[rrd]                                                   &  & 13 \arrow[color = red,rruu] \arrow[color = red,rru] \arrow[color = red,rrd]                      &  & 7  &                                                   \\
\begin{smallmatrix} 4 \\ 5 \\ 6 \end{smallmatrix}  \arrow[rrr, bend left = 25, start anchor = {[yshift=0.5ex]}]  \arrow[rrr, bend left= 25, start anchor = {[yshift=-1ex]}] \arrow[rrr, bend left= 25, start anchor = {[yshift=-2.5ex]}] & 17 \arrow[rruu] \arrow[rru] \arrow[rr] \arrow[rrd] \arrow[rrdd] &  & 19 \arrow[rruu] \arrow[rrdd] \arrow[rr, shift left] \arrow[rr, shift right] &  & 18 \arrow[color = red,rru] \arrow[color = red,rr] \arrow[color = red,rrd] \arrow[color = red,rrr, bend left = 25, end anchor = {[yshift=0.5ex]}]  \arrow[color = red,rrr, bend left= 25, end anchor = {[yshift=-1ex]}] \arrow[color = red,rrr, bend left= 25, end anchor = {[yshift=-2.5ex]}] &  & 17 & \begin{smallmatrix} 4 \\ 5 \\ 6 \end{smallmatrix} \\
                                                                         & 8 \arrow[rru] \arrow[rr]                                        &  & 12 \arrow[rru] \arrow[rr]                                                   &  & 15 \arrow[color = red,rru] \arrow[color = red,rrd] \arrow[color = red,rrdd]                      &  & 8  &                                                   \\
                                                                         & 1 \arrow[rr]                                                    &  & 16 \arrow[rru] \arrow[rr] \arrow[rrd]                                       &  & 11 \arrow[color = red,rruu] \arrow[color = red,rru]                                  &  & 1  &                                                   \\
                                                                         & 3 \arrow[rru]                                                   &  &                                                                             &  & 13 \arrow[color = red,rr]                                                &  & 3  &                                                  
\end{tikzcd}
  \]
  The labels are again corresponding to the enumeration of the irreducible representations of $G$ as computed by GAP.  
\end{Exp}

Next, we give a negative example. 

\begin{Exp}\label{Exp: Negative example}
    To see a negative example, consider the group $A = \langle \frac{1}{2}(1,1,0) \rangle$ of order $2$. Taking $G = \langle A, T \rangle$ of type (C), we see immediately that the group $N \trianglelefteq G$ of diagonal matrices is $C_2 \times C_2$, whose order is not divisible by $3$. Let us draw the quiver $Q_N$ and then skew it by $T$ to see the loop appearing.  
    \[ 
    \begin{tikzpicture}[-Stealth]
        \begin{scope}[rotate=-15.4,inner sep=1.5mm]
        \begin{scriptsize}

        \coordinate (O) at (0, 0, 0);
        \coordinate (X) at (1,0,-1);
        \coordinate (Y) at (0, -1, 1);
        \coordinate (Z) at (-1, 1, 0);

        \node (A0) at (O) {(0,0)};
        \node (A1) at ($(X)$) {(1,1)};
        \node (A2) at ($(Y)+(Z)$) {(1,1)};
        \node (A3) at ($(Y)$) {(1,0)};
        \node (A4) at ($(X)+(Z)$) {(1,0)};
        \node (A5) at ($(Z)$) {(0,1)};
        \node (A6) at ($(X)+(Y)$) {(0,1)};
        
        \draw (A0)--(A1);
        \draw (A0)--(A3);
        \draw (A0)--(A5);
        \draw (A1)--(A6);
        \draw (A1)--(A4);
        \draw (A2)--(A0);
        \draw (A3)--(A2);
        \draw (A3)--(A6);
        \draw (A4)--(A0);
        \draw (A5)--(A4);
        \draw (A5)--(A2);
        \draw (A6)--(A0);
        
     \end{scriptsize}
        \end{scope}
 \end{tikzpicture}
    \] 
    Clearly the $T$-orbits are $\{ (0,0) \}$ and $o = \{ (1,0), (0,1), (1,1) \}$. The skewed quiver is given by $3$ vertices for the stabilised point $(0,0)$ and $1$ vertex $o$ for the non-trivial orbit. Furthermore, the orbit has $2$ loops.
    \[
    \begin{tikzcd}
                           & x_2 \arrow[d, shift right]                                             &                            \\
x_1 \arrow[r, shift right] & o  \arrow[loop, in = 225, out= 260, distance = 2em] \arrow[loop, out = 270, in= 305, distance = 2em] \arrow[l, shift right] \arrow[u, shift right] \arrow[r, shift right] & x_3 \arrow[l, shift right]
\end{tikzcd}
    \]
    To construct a negative example for a group of type (D), we consider the same group $A$, and add the matrix $r = \left( \begin{smallmatrix}
        0 & -1 & 0 \\ 1 & 0 & 0 \\ 0 & 0 & 1
    \end{smallmatrix} \right).$ This gives rise to $G = \langle A, T, r \rangle \simeq S_4$, which is the special case we treated separately in the proof of \Cref{lem: 3 doesn't divide}, so we draw the quiver here. The vertex $(0,0)$ is stabilised by all of $S_3$. Since $S_3$ has $3$ irreducible representations, the stabilised vertex splits into $3$ vertices. The orbit of order $3$ now has a stabiliser of size $2$ attached, so the orbit gives rise to $2$ vertices in the quiver, so we draw the vertices in two rows corresponding to the two orbits. 
    \[ 
    \begin{tikzcd}
1 \arrow[rd, shift right] &                                                                                                                       & 3 \arrow[ld, shift right] \arrow[rd, shift right] &                                                                                                                       & 2 \arrow[ld, shift right] \\
                          & 5 \arrow[lu, shift right] \arrow[ru, shift right] \arrow[rr, shift right] \arrow[loop below] &                                                   & 4 \arrow[lu, shift right] \arrow[ru, shift right] \arrow[ll, shift right] \arrow[loop below] &                          
\end{tikzcd}
    \]
\end{Exp}

As noted, many of the involved $2$-repre\-sen\-ta\-tion infinite algebras are levelled. We record some easy observations that may be useful for future investigation of this phenomenon, in particular with respect to the effect of $2$-Auslander-Platzeck-Reiten tilting.

\begin{Rem}
    The $2$-representation infinite algebras we obtain by cutting the McKay quiver are \emph{levelled algebras}, meaning that the cut quiver is the Hasse diagram of a \emph{ranked poset}. It is clear that the cut quiver for the diagonal subgroup $N$ is $2$-levelled, and that actions of groups which leave the cut invariant give rise to automorphisms of the $2$-levelled quiver. It is also not difficult to see that the resulting skewed quivers are again levelled. Furthermore, if the dimensions of the irreducible representations appearing in a cut McKay quiver $Q$ are known, and if the skewed quiver $Q \ast K$ is again a McKay quiver, we can deduce the possible dimensions of the irreducibles appearing in a level of $Q \ast K$. 
\end{Rem}

\subsection{Skewing and unskewing}\label{SSec: Skewing and unskewing}
In type (C), we can establish more than just the existence of cuts, using the following observation that goes back to the work of Reiten and Riedtmann \cite[Section 5.1]{ReitenRiedtmann}. 

\begin{Rem}
    Let $G = N \rtimes T$ be of type (C). Then $T$ acts on $Q_N$, coming from the action of $T$ on $\Bbbk N$ and on $R$. Note that $\hat{T} \simeq T$ acts on $\Bbbk N$ and on $R$ in the corresponding dual representations. It then follows that skewing again by $\hat{T}$ unskews the action of $T$, and the resulting algebra is Morita equivalent to $R \ast N$, i.e.\ we have
    \[ (R \ast G) \ast \hat{T} = (R \ast (N \rtimes T)) \ast \hat{T} \simeq  ((R \ast N) \ast T) \ast \hat{T} \simeq_M (R \ast N) \otimes_k (kT \otimes k\hat{T}) \simeq_M R\ast N. \]
    A more general version for solvable groups was proven in \cite[Proposition 5.3]{ReitenRiedtmann}.
\end{Rem}

\begin{Pro}
    Let $G = N \rtimes T$ of type (C). Then there is a bijection between the $T$-invariant cuts of $Q_N$ and the $\hat{T}$-invariant cuts of $Q_G$. The bijection is given by 
    \begin{align*}
        \{C \subseteq (Q_N)_1 \mid  C \text{ is } T\text{-invariant}  \} &\leftrightarrow \{C' \subseteq (Q_G)_1 \mid  C' \text{ is } \hat{T}\text{-invariant}  \} \\
        C &\mapsto C \ast T \\
        C' \ast \hat{T} &\mapsfrom C'.
          \end{align*} 
\end{Pro}

\begin{proof}
    Applying \Cref{Cor: Skewing by graded auts} twice, we see that $ C \mapsto (C \ast T)$ and $C' \mapsto C' \ast \hat{T}$ are well-defined. It suffices to check that $C = C \ast T \ast \hat{T}$, which is immediate since both $T$ and $\hat{T}$ act by graded automorphisms. 
\end{proof}

It is important to point out that both $Q_N$ and $Q_G$ also admit other higher preprojective cuts which are not invariant under the actions of $T$ and $\hat{T}$ respectively. In the following, we make use of ``cut mutation'', which corresponds to $2$-APR tilting, see for example \cite{DramburgGasanova} for the case when the group is abelian. The details of the procedure are not necessary, so the unfamiliar reader can simply take the cuts as given and note that turning a sink in a cut quiver into a source produces a new cut. 

\begin{Exp}
    We return to the group of type (C) from \Cref{Exp: Positive case for (C)}. This was $G \simeq (C_3 \times C_3) \rtimes T$. Note that $Q_N$ admits many cuts which are not $T$-invariant. We give such an example below by turning the vertex $(0,2)$ from a sink in the cut quiver into a source.  
\[ 
    \begin{tikzpicture}[-Stealth]
        \begin{scope}[rotate=-15.4,inner sep=1.5mm]
        \begin{scriptsize}

        \coordinate (O) at (0, 0, 0);
        \coordinate (X) at (1,0,-1);
        \coordinate (Y) at (0, -1, 1);
        \coordinate (Z) at (-1, 1, 0);

        \node (A0) at (O) {(0,0)};
        \node (A6) at ($1*(X)$) {(1,2)};
        \node (A3) at ($2 *(X)$) {(2,1)};
        \node (A5) at ($(X)+(Y)$) {(0,2)};
        \node (A4) at ($2*(X)+(Y)$) {(1,1)};
        \node (A1) at ($-1*(Y)$) {(1,0)};
        \node (A2) at ($(X)-1*(Y)$) {(2,2)};
        \draw[red] (A1)--(A0);
        \draw[red] (A1)--(A2);
        \draw[red] (A3)--(A2);
        \draw[red] (A3)--(A4);
        \draw (A5)--(A4);
        \draw (A5)--(A0);
        \draw (A0)--(A6);
        \draw (A6)--(A1);
        \draw (A2)--(A6);
        \draw (A6)--(A3);
        \draw (A4)--(A6);
        \draw[red] (A6)--(A5);

        \node(B0) at ($-1*(Y)-2*(X)$) {(2,2)}; 
        \node (B6) at ($-1*(Y)-1*(X)$) {(0,1)};
        \node (B3) at ($-1*(Y)$) {(1,0)};
        \node (B5) at ($-1*(X)$) {(2,1)};
        \node (B4) at ($(O)$) {(0,0)};
        \node (B1) at ($-2*(Y)-2*(X)$) {(0,2)};
        \node (B2) at ($-2*(Y)-1*(X)$) {(1,1)};
        \draw (B1)--(B0);
        \draw (B1)--(B2);
        \draw[red] (B3)--(B2);
        \draw[red] (B3)--(B4);
        \draw[red] (B5)--(B4);
        \draw[red] (B5)--(B0);
        \draw (B0)--(B6);
        \draw[red] (B6)--(B1);
        \draw (B2)--(B6);
        \draw (B6)--(B3);
        \draw (B4)--(B6);
        \draw (B6)--(B5);

        \node (C0) at ($-1*(X)+1*(Y)$) {(1,1)};
        \node (C6) at ($1*(Y)$) {(2,0)};
        \node (C3) at ($1*(X)+1*(Y)$) {(0,2)};
        \node (C5) at ($2*(Y)$) {(1,0)};
        \node (C4) at ($1*(X)+2*(Y)$) {(2,2)};
        \node (C1) at ($-1*(X)$) {(2,1)};
        \node (C2) at ($(O)$) {(0,0)};
        \draw[red] (C1)--(C0);
        \draw[red] (C1)--(C2);
        \draw (C3)--(C2);
        \draw (C3)--(C4);
        \draw[red] (C5)--(C4);
        \draw[red] (C5)--(C0);
        \draw (C0)--(C6);
        \draw (C6)--(C1);
        \draw (C2)--(C6);
        \draw[red] (C6)--(C3);
        \draw (C4)--(C6);
        \draw (C6)--(C5);
        
     \end{scriptsize}
        \end{scope}
 \end{tikzpicture}
    \] 
    Therefore, this cut does not give rise to a skewed cut for the quiver $Q_G$. However, if we perform the same operation at all vertices in orbit of $(0,2)$, we do obtain a new $T$-invariant cut. 
    \[ 
    \begin{tikzpicture}[-Stealth]
        \begin{scope}[rotate=-15.4,inner sep=1.5mm]
        \begin{scriptsize}

        \coordinate (O) at (0, 0, 0);
        \coordinate (X) at (1,0,-1);
        \coordinate (Y) at (0, -1, 1);
        \coordinate (Z) at (-1, 1, 0);

        \node (A0) at (O) {(0,0)};
        \node (A6) at ($1*(X)$) {(1,2)};
        \node (A3) at ($2 *(X)$) {(2,1)};
        \node (A5) at ($(X)+(Y)$) {(0,2)};
        \node (A4) at ($2*(X)+(Y)$) {(1,1)};
        \node (A1) at ($-1*(Y)$) {(1,0)};
        \node (A2) at ($(X)-1*(Y)$) {(2,2)};
        \draw (A1)--(A0);
        \draw (A1)--(A2);
        \draw (A3)--(A2);
        \draw (A3)--(A4);
        \draw (A5)--(A4);
        \draw (A5)--(A0);
        \draw (A0)--(A6);
        \draw[red] (A6)--(A1);
        \draw (A2)--(A6);
        \draw[red] (A6)--(A3);
        \draw (A4)--(A6);
        \draw[red] (A6)--(A5);

        \node(B0) at ($-1*(Y)-2*(X)$) {(2,2)}; 
        \node (B6) at ($-1*(Y)-1*(X)$) {(0,1)};
        \node (B3) at ($-1*(Y)$) {(1,0)};
        \node (B5) at ($-1*(X)$) {(2,1)};
        \node (B4) at ($(O)$) {(0,0)};
        \node (B1) at ($-2*(Y)-2*(X)$) {(0,2)};
        \node (B2) at ($-2*(Y)-1*(X)$) {(1,1)};
        \draw (B1)--(B0);
        \draw (B1)--(B2);
        \draw (B3)--(B2);
        \draw (B3)--(B4);
        \draw (B5)--(B4);
        \draw (B5)--(B0);
        \draw (B0)--(B6);
        \draw[red] (B6)--(B1);
        \draw (B2)--(B6);
        \draw[red] (B6)--(B3);
        \draw (B4)--(B6);
        \draw[red] (B6)--(B5);

        \node (C0) at ($-1*(X)+1*(Y)$) {(1,1)};
        \node (C6) at ($1*(Y)$) {(2,0)};
        \node (C3) at ($1*(X)+1*(Y)$) {(0,2)};
        \node (C5) at ($2*(Y)$) {(1,0)};
        \node (C4) at ($1*(X)+2*(Y)$) {(2,2)};
        \node (C1) at ($-1*(X)$) {(2,1)};
        \node (C2) at ($(O)$) {(0,0)};
        \draw (C1)--(C0);
        \draw (C1)--(C2);
        \draw (C3)--(C2);
        \draw (C3)--(C4);
        \draw (C5)--(C4);
        \draw (C5)--(C0);
        \draw (C0)--(C6);
        \draw[red] (C6)--(C1);
        \draw (C2)--(C6);
        \draw[red] (C6)--(C3);
        \draw (C4)--(C6);
        \draw[red] (C6)--(C5);
        
     \end{scriptsize}
        \end{scope}
 \end{tikzpicture}
    \] 
    and the skewed cut on $Q_G$ can be obtained by turning the vertex $o_1 =(0,2)^T $ from a sink into a source.
    \[ 
    \begin{tikzcd}[row sep= tiny]
x_1 \arrow[rrddddd] &  &                                                                    &  &                                                                                                                                  &  & x_1 \\
x_2 \arrow[rrdddd]  &  &                                                                    &  &                                                                                                                                  &  & x_2 \\
x_3 \arrow[rrddd]   &  &                                                                    &  &                                                                                                                                  &  & x_3 \\
                    &  &                                                                    &  &                                                                                                                                  &  &     \\
y_1 \arrow[rrd]     &  &                                                                    &  &                                                                                                                                  &  & y_1 \\
y_2 \arrow[rr]      &  & o_2  \arrow[color = red,rr, shift left=2] \arrow[color = red,rr, shift right=2] \arrow[color = red,rr] &  & o_1 \arrow[rruuuuu] \arrow[rruuuu] \arrow[rruuu] \arrow[rru] \arrow[rr] \arrow[rrd] \arrow[rrddd] \arrow[rrdddd] \arrow[rrddddd] &  & y_2 \\
y_3 \arrow[rru]     &  &                                                                    &  &                                                                                                                                  &  & y_3 \\
                    &  &                                                                    &  &                                                                                                                                  &  &     \\
z_1 \arrow[rruuu]   &  &                                                                    &  &                                                                                                                                  &  & z_1 \\
z_2 \arrow[rruuuu]  &  &                                                                    &  &                                                                                                                                  &  & z_2 \\
z_3 \arrow[rruuuuu] &  &                                                                    &  &                                                                                                                                  &  & z_3
\end{tikzcd}
    \]
    Similarly, one can find cuts on $Q_G$ which are not $\hat{T}$-invariant, but if we for example turn all three vertices coming from a stabilised point from source to sink, we again obtain a $\hat{T}$-invariant cut, and we can see this reflected in the quiver for $N$. 
\end{Exp}

\begin{Rem}
    The procedure we performed in the example above is called \emph{cut mutation}, which corresponds to $2$-APR tilting of $2$-representation infinite algebras. In \cite{DramburgGasanova2}, we showed that for McKay quivers $Q_N$ of abelian groups $N$, the mutation class of a given cut can naturally be endowed with the structure of a finite distributive lattice. Note that for groups of type (C), certain elements in the lattice for $Q_N$ correspond to cuts for $Q_G$, and one can pass between them by ``invariant mutation'', i.e. by performing successive mutations at all vertices lying in a given orbit. It would therefore be interesting to see whether the cuts for $Q_G$ can also be endowed with a lattice structure, and how this structure is related to the one for the mutation class on $Q_N$. We are currently not aware of a proof nor of a counterexample that mutation is transitive on the set of all cuts of $Q_G$. While this is not true for quivers $Q_N$ of abelian groups $N$, it may be expected to hold for groups of type (C) since the existence of cuts for $Q_G$ is tied to the existence of one specific mutation class for $Q_N$. 
\end{Rem}

Producing the same correspondence between cuts on $R \ast N$ and $R \ast G$ in type (D) is slightly more involved. However, since $K \simeq S_3 = C_3 \rtimes C_2$ is solvable in this case, one can recover 
\[ ((R \ast G) \ast \hat{K}) \ast \hat{C_3} \simeq_M (R \ast (N \rtimes C_3)) \ast \hat{C_3} \simeq_M R \ast N  \]
and transfer cuts as long as all intermediate cuts are invariant under the respective actions.

\section*{Acknowledgements}
We thank Martin Herschend and Mads Hustad Sandøy for many helpful discussions and comments. The second author was supported by a fellowship from the Wenner-Gren Foundations (grant WGF2022-0052).

\printbibliography

@article{HIO,
title = {n-representation infinite algebras},
journal = {Advances in Mathematics},
volume = {252},
pages = {292-342},
year = {2014},
issn = {0001-8708},
doi = {https://doi.org/10.1016/j.aim.2013.09.023},
author = {Martin Herschend and Osamu Iyama and Steffen Oppermann}
}

@article{BSW,
title = {Superpotentials and higher order derivations},
journal = {Journal of Pure and Applied Algebra},
volume = {214},
number = {9},
pages = {1501-1522},
year = {2010},
issn = {0022-4049},
doi = {https://doi.org/10.1016/j.jpaa.2009.07.013},
author = {Raf Bocklandt and Travis Schedler and Michael Wemyss}
}

@article{Thibault,
author = {Thibault, Louis-Philippe},
year = {2016},
month = {03},
pages = {},
title = {Preprojective algebra structure on skew-group algebras},
volume = {365},
journal = {Advances in Mathematics},
doi = {10.1016/j.aim.2020.107033}
}

@article{AIR,
 ISSN = {00029327, 10806377},
 URL = {https://www.jstor.org/stable/26571360},
 author = {Claire Amiot and Osamu Iyama and Idun Reiten},
 journal = {American Journal of Mathematics},
 number = {3},
 pages = {813--857},
 publisher = {The Johns Hopkins University Press},
 title = {Stable categories of Cohen-Macaulay modules and cluster categories},
 volume = {137},
 year = {2015}
}

@article{LeMeur,
title = {Crossed products of Calabi-Yau algebras by finite groups},
journal = {Journal of Pure and Applied Algebra},
volume = {224},
number = {10},
pages = {106394},
year = {2020},
issn = {0022-4049},
doi = {https://doi.org/10.1016/j.jpaa.2020.106394},
author = {Patrick {Le Meur}},
}

@article{IyamaOppermannStable,
title = {Stable categories of higher preprojective algebras},
journal = {Advances in Mathematics},
volume = {244},
pages = {23-68},
year = {2013},
issn = {0001-8708},
doi = {https://doi.org/10.1016/j.aim.2013.03.013},
author = {Osamu Iyama and Steffen Oppermann},
}

@article{Giovannini,
author = {Simone Giovannini},
title = {Higher representation infinite algebras from McKay quivers of metacyclic groups},
journal = {Communications in Algebra},
volume = {47},
number = {9},
pages = {3672-3711},
year  = {2019},
publisher = {Taylor & Francis},
doi = {10.1080/00927872.2019.1570230}
}

@article{ReitenRiedtmann,
title = {Skew group algebras in the representation theory of artin algebras},
journal = {Journal of Algebra},
volume = {92},
number = {1},
pages = {224-282},
year = {1985},
issn = {0021-8693},
doi = {https://doi.org/10.1016/0021-8693(85)90156-5},
author = {Idun Reiten and Christine Riedtmann}
}

@incollection{MR0604577,
    AUTHOR = {McKay, John},
     TITLE = {Graphs, singularities, and finite groups},
 BOOKTITLE = {The {S}anta {C}ruz {C}onference on {F}inite {G}roups ({U}niv.
              {C}alifornia, {S}anta {C}ruz, {C}alif., 1979)},
    SERIES = {Proc. Sympos. Pure Math.},
    VOLUME = {37},
     PAGES = {183--186},
 PUBLISHER = {Amer. Math. Soc., Providence, RI},
      YEAR = {1980},
      ISBN = {0-8218-1440-0},
   MRCLASS = {20C99 (05C25 14B05 17B10)},
  MRNUMBER = {604577},
}

@article{DramburgGasanova,
title = {The 3-preprojective algebras of type Ã},
journal = {Journal of Pure and Applied Algebra},
volume = {228},
number = {12},
pages = {107760},
year = {2024},
issn = {0022-4049},
doi = {https://doi.org/10.1016/j.jpaa.2024.107760},
url = {https://www.sciencedirect.com/science/article/pii/S0022404924001579},
author = {Darius Dramburg and Oleksandra Gasanova}
}

@misc{DramburgSandoy,
      title={On compatibility of Koszul- and higher preprojective gradings}, 
      author={Darius Dramburg and Mads Hustad Sandøy},
      year={2025},
      eprint={2411.13283},
      archivePrefix={arXiv},
      primaryClass={math.RT},
      url={https://arxiv.org/abs/2411.13283}, 
}

@misc{DramburgGasanova2,
      title={A classification of $n$-representation infinite algebras of type \~A}, 
      author={Darius Dramburg and Oleksandra Gasanova},
      year={2024},
      eprint={2409.06553},
      archivePrefix={arXiv},
      primaryClass={math.RT},
      url={https://arxiv.org/abs/2409.06553}, 
}

@book{YauYu,
  title={Gorenstein quotient singularities in dimension three},
  author={Yau, Stephen Shing-Toung and Yu, Yung},
  volume={505},
  year={1993},
  publisher={American Mathematical Soc.}
}

@article{Demonet,
  title={Skew group algebras of path algebras and preprojective algebras},
  author={Demonet, Laurent},
  journal={Journal of Algebra},
  volume={323},
  number={4},
  pages={1052--1059},
  year={2010},
  publisher={Elsevier}
}

@article{NoLoops,
title = {A proof of the strong no loop conjecture},
journal = {Advances in Mathematics},
volume = {228},
number = {5},
pages = {2731-2742},
year = {2011},
issn = {0001-8708},
doi = {https://doi.org/10.1016/j.aim.2011.06.042},
url = {https://www.sciencedirect.com/science/article/pii/S0001870811002714},
author = {Kiyoshi Igusa and Shiping Liu and Charles Paquette}
}

@manual{GAP4,
    organization = "The GAP~Group",
    title        = "{GAP -- Groups, Algorithms, and Programming,
                    Version 4.14.0}",
    year         = 2024,
    url          = {https://www.gap-system.org}
}

@article{GradedMorita,
    AUTHOR = {Abrams, Gene and Ruiz, Efren and Tomforde, Mark},
     TITLE = {Morita equivalence for graded rings},
   JOURNAL = {J. Algebra},
  FJOURNAL = {Journal of Algebra},
    VOLUME = {617},
      YEAR = {2023},
     PAGES = {79--112},
      ISSN = {0021-8693,1090-266X},
   MRCLASS = {16D90 (16S50 16S88 16W50)},
  MRNUMBER = {4513781},
MRREVIEWER = {Blas\ Torrecillas Jover},
       DOI = {10.1016/j.jalgebra.2022.10.036},
       URL = {https://doi.org/10.1016/j.jalgebra.2022.10.036},
}

@article{ItoTrihedral,
  title={Crepant resolution of trihedral singularities and the orbifold Euler characteristics},
  author={Ito, Yukari},
  journal={International Journal of Mathematics},
  volume={6},
  number={1},
  pages={33--44},
  year={1995},
  publisher={Singapore: World Scientific, c1990-}
}

@article{DonovanFreislich,
  title={The representation theory of finite graphs and associated algebras},
  author={Donovan, Peter and Freislich, Mary Ruth},
  journal={Carleton Lecture Notes Nr. 5},
  year={1973}
}

@article{nazarova1973representations,
  title={Representations of quivers of infinite type},
  author={Nazarova, Liudmila A.},
  journal={Mathematics of the USSR-Izvestiya},
  volume={7},
  number={4},
  pages={749},
  year={1973},
  publisher={IOP Publishing}
}

@misc{Dramburg1,
      title={2-representation infinite algebras from non-abelian subgroups of SL(3, C). Part II: Central extensions and exceptionals}, 
      author={Darius Dramburg},
      year={2025},
      %eprint={2409.06553},
      %archivePrefix={arXiv},
      %primaryClass={math.RT},
      %url={https://arxiv.org/abs/2409.06553}, 
}

@misc{TomonagaSiltingMutations,
      title={On silting mutations preserving global dimension}, 
      author={Ryu Tomonaga},
      year={2025},
      eprint={2510.26206},
      archivePrefix={arXiv},
      primaryClass={math.RT},
      url={https://arxiv.org/abs/2510.26206}, 
}

\end{document}